\documentclass[12pt,a4paper]{amsart}

\usepackage{amsthm}
\usepackage{amssymb}
\usepackage{amsmath}
\usepackage{array}
\usepackage[english]{babel}
\usepackage{tikz-cd}
\usepackage{multirow}
\usepackage{booktabs}
\usetikzlibrary{arrows, shapes.geometric}
\usepackage{csquotes}
\usepackage[margin=2.7cm]{geometry}
\usepackage[backend=biber,style=numeric]{biblatex}
\usepackage{subcaption}
\usepackage{caption}
\usepackage{graphicx}
\usepackage{hyperref}
\hypersetup{hidelinks}
\usepackage{float}

\theoremstyle{plain}
\newtheorem{thm}{Theorem}[section]

\newtheorem{prop}[thm]{Proposition}
\newtheorem{cor}[thm]{Corollary}

\theoremstyle{definition}
\newtheorem{defn}[thm]{Definition}
\newtheorem{ex}[thm]{Example}

\theoremstyle{remark}
\newtheorem{rem}[thm]{Remark}

\newcommand{\cA}{{\mathcal A}}
\newcommand{\cD}{{\mathcal D}}
\newcommand{\cF}{{\mathcal F}}
\newcommand{\cG}{{\mathcal G}}
\newcommand{\cM}{{\mathcal M}}
\newcommand{\cN}{{\mathcal N}}
\newcommand{\cX}{{\mathcal X}}

\newcommand{\cXtop}{{{\mathcal X}^\top}}

\newcommand{\N}{{\mathbb{N}}}
\newcommand{\R}{{\mathbb{R}}}
\newcommand{\T}{{\mathbb{T}}}
\newcommand{\Z}{{\mathbb{Z}}}

\newcommand{\mvmap}{\rightrightarrows}
\newcommand{\walk}{{\rightsquigarrow}}
\newcommand{\Int}{\mathop{\mathrm{int}}\nolimits}
\newcommand{\cl}{\mathop{\mathrm{cl}}\nolimits}
\newcommand{\Inv}{\mathop{\mathrm{Inv}}\nolimits}
\newcommand{\Con}{\mathop{\mathrm{Con}}\nolimits}
\newcommand{\sA}{{\mathsf{ A}}}
\newcommand{\sJ}{{\mathsf{ J}}}
\newcommand{\sL}{{\mathsf{ L}}}
\newcommand{\sN}{{\mathsf{ N}}}
\newcommand{\sO}{{\mathsf{ O}}}
\newcommand{\sP}{{\mathsf{ P}}}
\newcommand{\sAtt}{{\mathsf{ Att}}}
\newcommand{\sANbhd}{{\mathsf{ ANbhd}}}
\newcommand{\sABlock}{{\mathsf{ ABlock}}}
\newcommand{\sInvset}{{\mathsf{ Invset}}}
\newcommand{\sMG}{{\mathsf{ MG}}}
\newcommand{\sMR}{{\mathsf{ MR}}}
\newcommand{\sMD}{{\mathsf{ MD}}}
\newcommand{\sSC}{{\mathsf{ SCC}}}
\newcommand{\sRC}{{\mathsf{ RC}}}
\newcommand{\setof}[1]{\left\{ {#1}\right\}}
\newcommand{\setdef}[2]{\left\{{#1}\,\left|\,{#2}\right.\right\}}

\title[Characterizing High-dimensional Dynamics]{Characterizing High-dimensional Dynamics by Combinatorial-Topological Methods on a Latent Space}

\author[Bailon]{Patrick Bailon}
\address{Department of Mathematics, Rutgers University, Piscataway, NJ 08854, USA}
\email{plb103@scarletmail.rutgers.edu}

\author[Gameiro]{Marcio Gameiro}
\address{Department of Mathematics, Rutgers University, Piscataway, NJ 08854, USA}
\email{gameiro@math.rutgers.edu}

\author[Gelb]{Brittany Gelb}
\address{Department of Mathematics, Rutgers University, Piscataway, NJ 08854, USA}
\email{brittany.gelb@rutgers.edu}

\author[Kalies]{William Kalies}
\address{Department of Mathematics and Statistics, University of Toledo, Toledo, OH 43606, USA}
\email{william.kalies@utoledo.edu}

\author[Kramar]{Miroslav Kramar}
\address{Department of Mathematics, University of Oklahoma, Norman, OK 73019, USA}
\email{miro@ou.edu}

\author[Mischaikow]{Konstantin Mischaikow}
\address{Department of Mathematics, Rutgers University, Piscataway, NJ 08854, USA}
\email{mischaik@math.rutgers.edu}

\author[Rivas]{Bernardo Rivas}
\address{Department of Mathematics and Statistics, University of Toledo, Toledo, OH 43606, USA}
\email{bernardo.dopradorivas@utoledo.edu}

\author[Vieira]{Ewerton Vieira}
\address{Department of Mathematics, Rutgers University, Piscataway, NJ 08854, USA}
\email{er691@rutgers.edu}


\begin{document}

\begin{abstract}
  Combinatorial-topological methods for characterizing dynamics are rigorous, generalizable, computable, and they only require approximations, but the dimension of the phase space is a computational bottleneck to their wider application. Motivated by the growing number of machine learning techniques for obtaining lower-dimensional latent representations of dynamics, we present an initial study of combinatorial-topological techniques in the dimensionality reduction setting. We establish bounds under which an algebraic structure that organizes dynamics can be lifted from the latent space to the original system. As a corollary, one can conclude the existence of attractors within certain regions of the original phase space. The hypothesis of these results is expressed in terms of an approximate semiconjugacy between the original and latent dynamics. In practice, verifying this hypothesis might be impossible or computationally intractable without further assumptions. Nonetheless, to demonstrate the ideas, we combine autoencoder-based models with Conley-Morse graph computations for Leslie population models, a thirteen-dimensional Mediterranean red coral population model, and the Chafee-Infante equation. While the lift of the Conley index is still an open question, the examples recover the expected algebraic topological invariants in several settings.
\end{abstract}
\maketitle

\section{Introduction}
\label{sec:introduction}

The purpose of this paper is to introduce a new approach for the analysis of high-dimensional data-driven dynamics.
The underlying assumption is that the data is generated by a continuous dynamical system $\varphi\colon \T^+ \times X\to X$ where $X$ is a complete metric space and $\T = \Z$ or $\R$.
For the sake of simplicity of exposition we assume a fixed sampling rate, and thus, observe the dynamics of $\varphi$ via a finite set of samples from a  continuous map $f\colon X \to X$.

To have any hope of using finite data to characterize the global dynamics of $f$ requires additional assumptions. For this paper we assume that $f$ possesses a \emph{global compact attractor} $A$, i.e., $A$ is a compact subset of $X$ with the property that $f(A)=A$ and $A$ attracts all bounded sets of $X$.
Understanding the dynamics of $f$ restricted to $A$ provides an understanding of the asymptotic dynamics of $f$. Our characterization is based on topological techniques, thus conceptually our goal is to understand the action of $f$ on open sets.
In this spirit, a choice of open covering of $A$ can be viewed as a choice of scale at which we wish to study the dynamics, and the compactness of $A$ allows us to assume that this covering is finite. Thus, if we identify data points with open sets, then compactness appears to be the minimal assumption that allows one to claim that the dynamics of $f$ can be characterized via a finite set of data points.

An alternative assumption is that there exists a finite-dimensional manifold upon which the dynamics can be predicted, modeled, or interpreted~\cite{lusch:kutz:brunton:2018, champion:lusch:kutz:brunton:2019, linot:graham:2022, floryan:graham:2022}.
Success of this approach depends on the implicit assumption that the finite-dimensional manifold is robust with respect to error and parameters. Abstractly, this is essentially equivalent to the existence of an \emph{inertial manifold}, i.e., a finite-dimensional, invariant, Lipschitz manifold that attracts all trajectories exponentially (see \cite{foias:sell:temam, constantin:foias:nicolaenko:temam, mallet-paret:sell}).

The existence of a global attractor is a relatively weak assumption that holds for a wide range of dynamical systems and is robust with respect to a wide range of perturbations (for a general discussion see \cite{hale1,hale2} and references therein). In contrast the existence of inertial manifolds is much more limited (see \cite{mallet-paret:sell:shao, zelik} and references therein). This suggests that it is safer to pursue the characterization of global dynamics by assuming the existence of a compact global attractor as opposed to the existence of a finite-dimensional manifold.

With this discussion in mind and for the sake of simplicity of exposition, for the remainder of this paper we assume that $f\colon X\to X$ and $g\colon Z\to Z$ are continuous functions on compact metric spaces $(X,d_X)$ and $(Z,d_Z)$, respectively.
Recall that $f$ and $g$ are \emph{semiconjugate} if there exists a surjective continuous map $E\colon X\to Z$ such that the following diagram commutes
\begin{equation}
  \label{eq:semiconjugacy}
  \begin{tikzcd}
    X \arrow{r}{f} \arrow[swap]{d}{E} & X \arrow{d}{E} \\
    Z \arrow{r}{g}  & Z.
  \end{tikzcd}
\end{equation}

Classical methods tend to focus on understanding the existence and  structure of \emph{invariant sets}, i.e., sets $S\subset X$ such that $f(S) = S$. The following theorem highlights the connection between $\sInvset(f)$ and $\sInvset(g)$, the sets of all invariant sets of $f$ and $g$, respectively.

\begin{thm}
  \label{thm:semiconjugacy}
  Let $X$ and $Z$ be compact metric spaces. Assume that $E\colon X\to Z$ is a continuous surjection and that the diagram \eqref{eq:semiconjugacy} commutes. Then
  \begin{enumerate}
    \item[(i)] If $S\in \sInvset(f)$, then $E(S)\in \sInvset(g)$.
    \item[(ii)] If $S\in \sInvset(g)$, then there exists $\hat{S}\in \sInvset(f)$ such that $E(\hat{S})=S$.
  \end{enumerate}
\end{thm}

Semiconjugacies, and more generally topological conjugacies,
can be approximated with autoencoder architectures from machine learning in which $E$ is an encoder and $g$ is a latent dynamics model \cite{bramburger:brunton:kutz:2021, carney:gonzaleztokman:kardkasem:zhang:2025, bevanda:kirmayr:sosnowski:hirche:2022, bizzi:nissenbaum:pereira:2025, bramburger:2024,kvalheim:sontag:autoencoding}.
The approach has also been used to estimate regions of attraction in the full state space in robotics \cite{morals,werner:etal:halo}.

There are at least two reasons why Theorem~\ref{thm:semiconjugacy} is of limited use in data-driven settings.
First, even explicit knowledge of $g$ does not by itself provide an explicit description of $\sInvset(g)$.
Second, the commutativity of \eqref{eq:semiconjugacy} is too much to expect; one can only hope for the existence of a reasonable bound on the \emph{residual},
\begin{equation}
  \label{eq:residual}
  \sup_{x\in X}d_Z\bigl(g(E(x)),E(f(x))\bigr).
\end{equation}
In this case we say that $f$ and $g$ are \emph{approximately semiconjugate}.
Bifurcation theory tells us that even in the ideal case in which $X=Z$ and $E$ is the identity, if $g\neq f$ then there is no obvious relationship between $\sInvset(f)$ and $\sInvset(g)$.

In this paper we do not focus on how best to learn an autoencoder and latent dynamics.
Rather, our goal is to identify what can be rigorously inferred once $E$ and $g$ have been provided. Our work is heavily motivated by \cite{gameiro:gelb:mischaikow}, which showcases the recovery of combinatorial Conley-Morse invariants for low-dimensional data-driven systems.
A key step in that paper is the identification of attracting blocks.

\begin{defn}
  \label{defn:attractingblock}
  A compact set $N\subset X$ is an \emph{attracting block} for $f$ if $f(N)\subset \Int(N)$ where $\Int$  denotes  interior.
  The \emph{tolerance} of $N$ with respect to $f$ is defined to be
  \[
    \tau_f(N) := \sup\setdef{\nu >0}{B_\nu(f(N))\subset N},
  \]
  where $B_\nu(f(N))$ is the open $\nu$-neighborhood of $f(N)$.
\end{defn}

Attracting blocks allow us to identify the existence of invariant sets.

\begin{defn}
  \label{def:attractor}
  A set $A\subset X$ is an \emph{attractor} if there exists an attracting block $N\subset X$ such that
  \begin{equation}
    \label{eq:att}
    A = \omega(N,f) := \bigcap_{n\geq 0}\cl\left(\bigcup_{k=n}^\infty f^k(N)\right) \subset \Int(N),
  \end{equation}
  where $\cl$ denotes closure.
\end{defn}

A standard result (see \cite{kalies:mischaikow:vandervorst:14}) is that an attractor is an invariant set.
Whereas bifurcation theory tells us that invariant sets, e.g., attractors, can be sensitive to small perturbations, attracting blocks are robust. More precisely, given an attracting block $N$ for $f$, if $\widetilde f\colon X\to X$ is continuous and $\sup_{x\in X}d_X(f(x),\widetilde f(x))<\tau_f(N)$, then $N$ is an attracting block for $\widetilde f$.
Lifting this observation to the setting of autoencoders and latent dynamics leads to the following theorem.

\begin{thm}
  \label{thm:main_simple}
  Let $(X,d_X)$ and $(Z,d_Z)$ be compact metric spaces. Suppose that the functions $f\colon X \to X$, $g \colon Z \to Z$, and $E\colon X \to Z$ are continuous. If $N$ is an attracting block for $g$ and
  \begin{equation}
    \label{eq:loc_semiconjugacy}
    \sup_{x\in E^{-1}(N)}d_Z\bigl(g(E(x)),E(f(x))\bigr)<\tau_g(N),
  \end{equation}
  then $E^{-1}(N)$ is an attracting block  for $f$.
\end{thm}

An immediate corollary of Theorem~\ref{thm:main_simple} is that $\omega(E^{-1}(N),f)$ is an attractor for $f$.
An important caveat is that given distinct attracting blocks $N$ and $N'$  for $g$, if $N \cap N' \neq \emptyset$ it is possible that $\omega(E^{-1}(N),f) = \omega(E^{-1}(N'),f)$.
In other words, finding distinct attracting blocks for $g$ does not necessarily imply that distinct attractors for $f$ have been identified.

Having introduced the motivation and theoretical results for our approach, we turn to the application of these ideas.
If one is working directly with data, then both $f$ and its phase space $X$ may be unknown.
For many problems of interest $X$ is high-dimensional, which makes an explicit global understanding of how $f$ acts on $X$ impractical (if it were available, the value of learning latent dynamics would be questionable).
What is assumed to be known are the learned continuous functions, that is, the encoder $E\colon X\to Z$ and the latent dynamics $g\colon Z\to Z$.

In this paper we analyze $g$ using fairly standard computational topological techniques \cite{arai:kalies:kokubu:mischaikow:oka:pilarczyk,bush:gameiro:harker:kokubu:mischaikow:obayashi:pilarczyk,bush:cowan:harker:mischaikow,kalies:mischaikow:vandervorst:05} that are encoded in the software \texttt{CMGDB} \cite{CMGDB}.
In particular, given $g$, \texttt{CMGDB} identifies attracting blocks for $g$, from which we can determine their tolerance $\tau$.
An important caveat is that the computational cost of \texttt{CMGDB} grows exponentially with the dimension of $Z$.
In practice \texttt{CMGDB} is computationally effective if $Z \subset \R^d$ for
$d$ small, and in this paper we limit ourselves to $d=1,2,3$.

Even with the knowledge of $\tau$, in practice determining the validity of  \eqref{eq:loc_semiconjugacy} is either impossible or computationally intractable.
Our perspective is that Theorem~\ref{thm:main_simple} provides a sufficient condition under which information about attractors derived from the known latent dynamics $g$ can be lifted to the unknown dynamics of $f$.
As we make clear in Section~\ref{sec:applications}, the residual bound \eqref{eq:loc_semiconjugacy} appears to be far from a necessary condition, and our calculations provide more information about the dynamics of $f$ than that guaranteed by Theorem~\ref{thm:main_simple}.
We return to this point in Section~\ref{sec:conclusions}.

In summary, we propose that the approach for the analysis of data-driven dynamics described in this paper be viewed as a numerical method.
Theorem~\ref{thm:main_simple} provides a sufficient condition for identifying dynamics of interest.
In practice, validating \eqref{eq:loc_semiconjugacy} is intractable, but nevertheless our computational tools are sufficiently robust to provide insight into systems that are otherwise difficult to ascertain.

The outline for this paper is as follows.
Section~\ref{sec:approximate_semiconjugacy} contains the proof of Theorem~\ref{thm:main_simple}.
Section~\ref{sec:background} reviews Conley theory and \texttt{CMGDB}, the software used to analyze the latent dynamics, and discusses how its computational results characterize the latent dynamics.

Section~\ref{sec:applications} consists of a sequence of examples that demonstrate that our approach effectively characterizes the latent dynamics $g$ and that this characterization provides relevant information about the  dynamics of $f$.
For each example we choose a dataset from which we compute an approximate semiconjugacy $E$ and latent dynamics map $g$.
For the sake of reproducibility the essential computational details are provided in an appendix.
However, we emphasize that the focus of this paper is not on obtaining $E$ or $g$; their optimization is problem specific and beyond the scope of this manuscript.
Our approach takes $E$ and $g$ as given.

As indicated above the results of Section~\ref{sec:applications} exceed our theoretical guarantees.
In Section~\ref{sec:conclusions} we discuss why the condition is conservative and how our approach might be improved.

\section{Attracting blocks under approximate semiconjugacy}
\label{sec:approximate_semiconjugacy}

This section is devoted to the proofs of our theoretical results presented in Section~\ref{sec:introduction}.
Recall that a \emph{full trajectory} of $x \in X$ under $f$ is a function $\gamma_x \colon \Z \to X$ that satisfies $\gamma_x(0) = x$ and $\gamma_x(n+1) = f(\gamma_x(n))$ for all $n \in \Z$.
For $N\subset X$, let $\Inv(N,f)$ denote the \emph{maximal invariant set} in $N$, i.e., the set of points in $N$ that lie on a full trajectory under $f$ contained in $N$.
To prove Theorem~\ref{thm:semiconjugacy} we use the following elementary proposition.

\begin{prop}
  \label{prop:invset}
  Let $f\colon X\to X$ be a continuous function and let $S\subset X$.
  Then, $S\in \sInvset(f)$  if and only if for every $x\in S$ there exists a full trajectory $\gamma_x\colon \Z\to S$.
\end{prop}
\begin{proof}
  Suppose $f(S)=S$ and fix $x\in S$. Set $x_0=x$. Since $f(S)=S$, we may inductively choose $x_{-n-1}\in S$ such that $f(x_{-n-1})=x_{-n}$, and for $n\geq 0$ choose $x_n=f^n(x)$. Then $\gamma_x(n)=x_n$ defines a full trajectory in $S$ through $x$. Conversely, if every $x\in S$ admits a full trajectory, then $f(x)=\gamma_x(1)\in S$, so $f(S)\subset S$, while $x=f(\gamma_x(-1))$, so $S\subset f(S)$.
\end{proof}

\begin{proof}[Proof of Theorem~\ref{thm:semiconjugacy}]
  (i) Assume $S\in \sInvset(f)$, so $f(S) = S$.
  We show that $g(E(S)) = E(S)$. Using the semiconjugacy and the invariance of $S$ we have
  \[
    g(E(S)) = E(f(S)) = E(S).
  \]
  Hence $E(S)\in \sInvset(g)$.

  \medskip

  (ii) Assume $S\in \sInvset(g)$, so $g(S) = S$. By Proposition~\ref{prop:invset}, every $z\in S$ lies on a full trajectory under $g$ in $S$. For each full trajectory $\gamma\colon\Z\to S$ and $n\in\N$ let
  \[
    K_n(\gamma) := f^n(E^{-1}(\gamma(-n))).
  \]
  The sets $K_n(\gamma)$ have the following properties:
  \begin{enumerate}
    \item[(i)] \textbf{Compactness:} Each fiber $E^{-1}(\gamma(-n))$ is a nonempty closed subset of $X$ because $E$ is continuous and surjective. Since $X$ is compact and $f$ is continuous, each $K_n(\gamma)$ is nonempty and compact.

    \item[(ii)] \textbf{Nestedness:} By the commutative diagram \eqref{eq:semiconjugacy},
      \begin{align*}
        E\bigl(f(E^{-1}(\gamma(-(n+1))))\bigr)
        &=g\bigl(E(E^{-1}(\gamma(-(n+1))))\bigr)\\
        &=g(\gamma(-(n+1)))=\gamma(-n).
      \end{align*}
      Hence $f(E^{-1}(\gamma(-(n+1)))) \subset E^{-1}(\gamma(-n))$. Applying $f^n$ to both sides gives $K_{n+1}(\gamma) \subset K_n(\gamma)$.

    \item[(iii)] \textbf{Intersection:} The sets $K_n(\gamma)$ form a nested sequence of nonempty compact sets. Hence, by the Finite Intersection Property, the intersection
      \[
        K(\gamma) := \bigcap_{n=0}^\infty K_n(\gamma)
      \]
      is nonempty. Since
      \begin{align*}
        E(K_n(\gamma)) &= E(f^n(E^{-1}(\gamma(-n))))\\
        &= g^n(E(E^{-1}(\gamma(-n))))\\
        &= g^n(\gamma(-n)) = \{\gamma(0)\}
      \end{align*}
      for all $n \in \N$, we have that $E(K(\gamma)) = \{\gamma(0)\}$.
  \end{enumerate}

  Now define
  \[
    \hat{S} := \bigcup_{\gamma} K(\gamma),
  \]
  where the union is taken over all full trajectories in $S$. Since every point of $S$ lies on a full trajectory,
  \[
    E(\hat{S}) = \bigcup_{\gamma} E(K(\gamma)) = \bigcup_{\gamma} \{\gamma(0)\} = S.
  \]
  For a full trajectory $\gamma$, let $\gamma^+(k):=\gamma(k+1)$. It follows that
  \[
    f(K_n(\gamma)) = f(f^n(E^{-1}(\gamma(-n)))) = f^{n+1}(E^{-1}(\gamma(-n))) = K_{n+1}(\gamma^+).
  \]
  If $y\in\bigcap_{n=0}^\infty f(K_n(\gamma))$, then the sets $f^{-1}(y)\cap K_n(\gamma)$ are nested nonempty compact sets. Their intersection is nonempty, so $y\in f(K(\gamma))$. Consequently,
  \[
    f(K(\gamma)) = \bigcap_{n=0}^\infty f(K_n(\gamma)) = \bigcap_{n=0}^\infty K_{n+1}(\gamma^+) = K(\gamma^+).
  \]
  Since $\gamma\mapsto\gamma^+$ permutes the full trajectories in $S$, it follows that $f(\hat{S})=\hat{S}$. Thus $\hat{S}$ is the required invariant set in $X$.
\end{proof}

We now turn to the main result, which relaxes the semiconjugacy requirement of Theorem~\ref{thm:semiconjugacy} to a residual bound on the preimage of the attracting block under consideration.

\begin{proof}[Proof of Theorem~\ref{thm:main_simple}]
  Let $N$ be an attracting block for $g$ satisfying \eqref{eq:loc_semiconjugacy} and let
  \[
    \rho:=\sup_{x\in E^{-1}(N)}d_Z\bigl(g(E(x)),E(f(x))\bigr).
  \]
  Since $\rho<\tau_g(N)$, there exists $\epsilon>\rho$ such that $B_\epsilon(g(N))\subset N$. By openness, $B_\epsilon(g(N))\subset\Int(N)$. Since $X$ is compact and $E$ is continuous, $E^{-1}(N)$ is compact. The choice of $\rho$ then yields
  \[
    E(f(E^{-1}(N))) \subset B_\epsilon(g(N)) \subset \Int(N).
  \]
  Thus,
  \[
    f(E^{-1}(N)) \subset E^{-1}(\Int(N)) \subset \Int(E^{-1}(N)),
  \]
  where the last inclusion follows from the continuity of $E$. Therefore, $E^{-1}(N)$ is an attracting block.
\end{proof}

As we mentioned in Section~\ref{sec:introduction}, Theorem~\ref{thm:main_simple} implies that, for sufficiently small semiconjugacy error, the preimage of each attracting block for $g$ contains a nonempty invariant set for $f$. To highlight this fact we formulate it as a corollary.

\begin{cor}
  \label{cor:nonempty_lift}
  Assume that $N\subset E(X)$ and $N\neq \emptyset$ satisfies the hypotheses of Theorem~\ref{thm:main_simple}. Then $\Inv(E^{-1}(N), f) \neq \emptyset$.
\end{cor}

\begin{proof}
  Since $N\subset E(X)$ and $N\neq\emptyset$, $E^{-1}(N)$ is nonempty. By Theorem~\ref{thm:main_simple}, it is an attracting block for $f$. Since it is compact, $\Inv(E^{-1}(N),f)=\omega(E^{-1}(N),f)\neq\emptyset$.
\end{proof}

\section{Conley theory and Morse representations}
\label{sec:background}

In this section we provide a brief review of the concepts used in the applications described in the next section.
Section~\ref{subsec:conley} discusses Conley theory \cite{conley:cbms} using the language of order theory as developed in \cite{kalies:mischaikow:vandervorst:14, kalies:mischaikow:vandervorst:15, kalies:mischaikow:vandervorst:21} and presents a greatly abridged version of the Conley index that is sufficient for this paper.
Motivated by the results of \cite{gameiro:gelb:mischaikow}, this is the framework through which we hope to characterize the dynamics of $f\colon X \to X$.

Section~\ref{subsec:cmgdb} presents the theoretical underpinning for the software \texttt{CMGDB} that is used to perform computations on the latent dynamics model $g\colon \R^d\to \R^d$.

\subsection{Conley theory}
\label{subsec:conley}
Let $f\colon X\to X$ be a continuous map on a compact metric space.
Throughout this section we set $S := \Inv(X,f)$.

A \emph{Morse representation} of $S$ under $f$ is a finite poset $\sMR := \setdef{M(p)}{p\in \sP,\leq_\sP}$ of mutually disjoint, nonempty, compact invariant subsets of $S$, called \emph{Morse sets}, with the property that for each $x\in S\setminus \bigcup_{p\in \sP}M(p)$ and each full trajectory $\gamma_x\colon \Z\to S$ there exist $p,p'\in \sP$ such that $p < p'$,
\begin{equation}
  \label{eq:MRorder}
  \omega(x,f)\subset M(p)\quad\text{and}\quad \alpha_{\gamma_x}(x,f) := \bigcap_{n<0} \cl\left( \bigcup_{k \leq n}\{\gamma_x(k)\}\right) \subset M(p').
\end{equation}

For much of this paper our goal is to identify Morse representations of $S$.
The motivation for the methodology of this paper is Theorem~\ref{thm:main_simple}, which is stated in terms of attracting blocks.
As we now describe, there is an equivalence between lattices of attractors, lattices of attracting blocks, and Morse representations.

A compact set $N\subset X$ is an \emph{attracting neighborhood} if $\omega(N,f)\subset\Int(N)$. Every attracting block is an attracting neighborhood, and every attractor characterized using an attracting neighborhood admits an attracting block \cite[Section~1.1 and Theorem~1.2]{kalies:kasti:vandervorst}. Thus Definition~\ref{def:attractor} can equivalently be stated using attracting neighborhoods.

\begin{defn}
  Let $A \subset S$ be an attractor of $f$.
  The \emph{dual repeller} of $A$ is defined by
  \[
    A^* := \setdef{x\in S}{\omega(x,f)\cap A = \emptyset}.
  \]
\end{defn}

Let $\sANbhd(f)$ and $\sABlock(f)$ denote the set of attracting neighborhoods and attracting blocks of $f$, respectively.
As is shown in \cite{kalies:mischaikow:vandervorst:14}, $\sANbhd(f)$ is a bounded distributive lattice under the operations $\vee = \cup$ and $\wedge = \cap$ with minimal element $\emptyset$ and maximal element $X$. The same holds for $\sABlock(f)$ \cite[Section~1.1]{kalies:kasti:vandervorst}.
Since every attracting block is an attracting neighborhood, $\sABlock(f)$ is a sublattice of $\sANbhd(f)$.

Let $\sAtt(f)$ denote the set of attractors of $f$.
Then $\sAtt(f)$ is a bounded distributive lattice under the operations $A_0\vee A_1 = A_0\cup A_1$ and $A_0\wedge A_1 = \omega(A_0\cap A_1,f)$, with minimal element $\emptyset$ and maximal element $S$.

By \cite[Corollary~3.6 and Proposition~4.3]{kalies:mischaikow:vandervorst:14}, the map
\[
  \omega(\cdot,f)=\Inv(\cdot,f)\colon \sANbhd(f) \to \sAtt(f)
\]
is a surjective lattice homomorphism. Since every attractor admits an attracting block, the restriction of $\omega(\cdot,f)$ to $\sABlock(f)$ remains surjective.

Recall that a lattice $\sL$ is a poset with partial order defined by $L_0 \leq_\sL L_1$ if $L_0\vee L_1 = L_1$.
If $\sL$ is a finite lattice, then an element $L\in \sL$ is \emph{join irreducible} if $L$ has a unique immediate predecessor with respect to $\leq_\sL$.
We denote the set of join-irreducible elements of $\sL$ by $\sJ^\vee(\sL)$, and given $L\in\sJ^\vee(\sL)$ we let $L^<$ denote its unique immediate predecessor.

The following proposition is a restatement of \cite[Theorem 5]{kalies:mischaikow:vandervorst:21}.

\begin{prop}
  \label{prop:MR=Att}
  A finite collection of nonempty, mutually disjoint compact invariant sets $\sMR = \setdef{M(p)}{p\in (\sP,\leq_\sP)}$ whose order relation satisfies \eqref{eq:MRorder} for all $x\in S\setminus \bigcup_{p\in \sP}M(p)$ is a Morse representation of $f$ if and only if there exists a finite sublattice $\sA\subset \sAtt(f)$ containing $\emptyset$ and $S$, together with an order isomorphism $p\mapsto A_p$ from $\sP$ to $\sJ^\vee(\sA)$, such that
  \[
    M(p)=A_p\cap(A_p^<)^*
  \]
  for every $p\in\sP$.
\end{prop}

We conclude this section with a short review of the Conley index.
For the purposes of this presentation it is sufficient to observe that if $N_0,N_1\in \sABlock(f)$ and $N_0\subset N_1$, then $f$ induces a  map on relative homology $f_*\colon H_*(N_1,N_0) \to H_*(N_1,N_0)$.
The Conley index of $K =\Inv(\cl(N_1\setminus N_0),f)$, denoted by $\Con_*(K)$,  is the shift-equivalence class of $f_*$ (see \cite{mischaikow:mrozek,mischaikow:weibel}).
In this paper we restrict our attention to homology with field coefficients.
This implies that $f_*$ is a linear map in which case shift equivalence is determined by the rational canonical form of $f_*$ ignoring the zero eigenvalues of $f_*$.
Thus, we write
\begin{equation}
  \label{eq:CIpoly}
  \Con_*(K) = (q_0(x),q_1(x),q_2(x), \ldots)
\end{equation}
where $q_i(x)$ is the characteristic polynomial of the induced map $f_i$ on $H_i(N_1,N_0)$, with zero eigenvalues omitted.

The following result is fundamental.
\begin{prop}
  \label{prop:CInot0}
  If the Conley index is not trivial, i.e., $\Con_*(K) \neq (0,0,0,\ldots)$, then $K\neq \emptyset$.
\end{prop}

\begin{rem}
  \label{rem:trivialCon}
  The converse of Proposition~\ref{prop:CInot0} is not true.
  More specifically it is possible that the shift-equivalence class of $f_*\colon H_*(N_1,N_0) \to H_*(N_1,N_0)$ is trivial, but $\Inv(\cl(N_1\setminus N_0),f) \neq \emptyset$.
\end{rem}

The Conley index can also be used to prove the existence of equilibria \cite{srzednicki:85,mccord:88}, periodic orbits \cite{mccord:mischaikow:mrozek}, heteroclinic orbits \cite{conley:cbms}, chaotic dynamics \cite{mischaikow:mrozek:95,szymczak:96, day:junge:mischaikow, day:frongillo}, and semiconjugacies onto nontrivial dynamics \cite{mischaikow:95, mccord:mischaikow:96, mccord:00}.

\subsection{\texttt{CMGDB}}
\label{subsec:cmgdb}

We assume that we are given a continuous map $g\colon Z \to Z$, where $Z\subset\R^d$, and restrict our attention to a rectangular region $B = \prod_{i=1}^d [a_i,b_i] \subset Z$.
The first step is to choose a level of discretization of $B$ to produce a cubical complex $\cX$ (see \cite{kaczynski:mischaikow:mrozek:04}).
We refer to the $d$-dimensional cubes of $\cX$ as top cells and denote them by $\cX^\top = \setdef{\xi \in \cX}{\dim(\xi) = d}$.

The software \texttt{CMGDB} produces a combinatorial model for the dynamics of $g$ that takes the form of a \emph{combinatorial multivalued map} $\cG\colon \cX^\top \mvmap \cX^\top$, i.e., for each $\xi \in \cX^\top$, $\emptyset \neq \cG(\xi) \subset \cX^\top$.
Equivalently, $\cG$ can be viewed as a directed graph where $\xi\to \xi'$ if and only if $\xi'\in \cG(\xi)$.
We make use of both perspectives.
For $\cN\subset\cX^\top$, write $|\cN|:=\bigcup_{\xi\in\cN}|\xi|\subset B$ for its geometric realization.

A \textit{walk} in $\cG$ of length $K$ from $\xi_0\in\cXtop$ to $\xi_K\in\cXtop$ is a sequence of vertices
\[
  (\xi_0,\ldots,\xi_K),\qquad \xi_k\in\cG(\xi_{k-1})\quad\text{for }1\leq k\leq K.
\]
We denote a walk from $\xi_0$ to $\xi_K$ by $\xi_0 \walk \xi_K$.
An equivalence relation $\sim$ on $\cXtop$ is given by $\xi \sim \xi'$ if there exist walks $\xi \walk \xi'$ and $\xi' \walk \xi$.

A directed graph is \textit{strongly connected} if there exists a walk between each pair of vertices, namely the quotient set $\cXtop/\sim$ consists of a single equivalence class. The \textit{strongly connected components} $\sSC(\cG)$ of a directed graph $\cG: \cXtop \mvmap \cXtop$ are the equivalence classes of $\cXtop/\sim$, and we denote the quotient map by
\begin{equation}
  \label{eqn:map_pi}
  \pi_\cG\colon \cXtop \to \sSC(\cG).
\end{equation}

By contracting the edges of $\cG$ that have incident vertices in the same strongly connected component, we obtain a graph $\bar{\cG}: \sSC(\cG) \mvmap \sSC(\cG)$ called the \emph{condensation graph}. Observe that $\bar{\cG}$ is acyclic, and therefore $\sSC(\cG)$ is a partially ordered set (poset) where $\zeta' \leq_{\sSC(\cG)} \zeta$ if $\zeta \walk \zeta'$ in $\bar{\cG}$.

To obtain a finer description of the dynamics of $g$, we focus on particular strongly connected components.
A \emph{recurrent component} is a strongly connected component whose induced subgraph contains at least one edge.
The collection of recurrent components is denoted by $\sRC(\cG)$.
Since it is a subset of $\sSC(\cG)$, $\left(\sRC(\cG),\leq_{\sSC(\cG)} \right)$ is a poset.
The poset of recurrent components $\left(\sRC(\cG),\leq_{\sSC(\cG)} \right)$ is called the \emph{Morse graph} of $\cG$ and denoted by $\sMG(\cG)$.

The Morse graph $\sMG(\cG)$ is the combinatorial analog of a Morse representation $\sMR := \setdef{M(p)}{p\in \sP,\leq_\sP}$ for a continuous function.
Throughout the paper, we refer to the collection $\setdef{|\pi^{-1}_\cG(\cM)|}{\cM \in \sMG(\cG)}$ as the regions of phase space corresponding to the Morse graph $\sMG(\cG)$.

We now turn to the analogous lattice structures. Recall that given a poset $(\sP, \leq)$, a subset $I \subset \sP$ is a \emph{downset} if $p\in I$ and $q \leq p$ implies that $q \in I$.
We denote the collection of downsets of a finite poset $\sP$ by $\sO(\sP)$ and remark that $\sO(\sP)$ is a finite distributive lattice with operations $\vee = \cup$ and $\wedge = \cap$, and minimal and maximal elements given by $\emptyset$ and $\sP$.

\begin{defn}
  \label{defn:invset+}
  A \emph{forward invariant set} of $\cG$ is a set $\cN\subset \cXtop$ such that $\cG(\cN)\subset\cN$.
  We denote the set of forward invariant sets of $\cG$ by $\sInvset^+(\cG)$.
\end{defn}

In practice, given $g$ and a reasonable discretization of $B\subset Z$, $\sInvset^+(\cG)$ is an enormous lattice.
Thus, we focus on the following much smaller collection of forward invariant sets.

\begin{defn}
  \label{defn:attcG}
  An \emph{attractor} of $\cG$ is a set $\cA\subset \cXtop$ such that $\cG(\cA) = \cA$.
  We denote the set of attractors of $\cG$ by $\sAtt(\cG)$.
\end{defn}

We leave the proof of the following proposition to the reader.

\begin{prop}
  \label{prop:omegaInv+}
  Let $\cN \in \sInvset^+(\cG)$.
  Then, there exists a unique $\cA\in \sAtt(\cG)$ and $n_\cN\in \Z^+$ such that
  \[
    \omega(\cN,\cG) := \cA = \cG^n(\cN)\quad\text{for all $n\geq n_\cN$.}
  \]
\end{prop}

As is shown in \cite{kalies:mischaikow:vandervorst:15}, $\sInvset^+(\cG) = \sO(\sSC(\cG))$ and $\sAtt(\cG) = \sO(\sRC(\cG))$.
Furthermore, $\sInvset^+(\cG)$ is a finite distributive lattice with operations $\vee = \cup$ and $\wedge =\cap$ and minimal and maximal elements $\emptyset$ and $\cXtop$.
$\sAtt(\cG)$  is a finite distributive lattice with operations $\vee = \cup$ and $\cA_0\wedge\cA_1$
given by $\omega(\cA_0\cap\cA_1,\cG)$, the maximal attractor in $\cA_0\cap\cA_1$.
The minimal element is $\emptyset$, but the maximal element is typically a strict subset of $\cXtop$.

As discussed in Section~\ref{subsec:conley}, Morse representations are related to finite sublattices of attractors. In the combinatorial setting there is an analogous relation between $\sMG(\cG)$ and the lattice of attractors $\sAtt(\cG)$.
In particular, there is a natural isomorphism $\eta \colon \sMG(\cG) \to \sJ^\vee(\sAtt(\cG))$ taking elements of $\sMG(\cG)$ to join-irreducible elements of $\sAtt(\cG)$. Given $\cM \in \sMG(\cG)$, $\eta(\cM)$ is join irreducible and thus has a unique immediate predecessor $\eta(\cM)^<$.

The Conley index of $\cM$, denoted by $\Con_*(\cM)$, is the shift-equivalence class of the induced map on homology
\begin{equation}
  \label{eq:CIG}
  \cG_* \colon H_*(\eta(\cM),\eta(\cM)^<) \to H_*(\eta(\cM),\eta(\cM)^<).
\end{equation}
As presented at the beginning of this section, $\cG \colon \cXtop \mvmap \cXtop$.
To compute homology requires use of the full cubical complex $\cX$ and requires that $\cG$ satisfy specific constraints.
The theoretical details can be found in \cite{kaczynski:mischaikow:mrozek:04, harker:mischaikow:mrozek:nanda, harker:kokubu:mischaikow:pilarczyk}.
For the purposes of this paper it is sufficient to remark that \texttt{CMGDB} extends $\cG$ as a combinatorial multivalued map on $\cXtop$ to a combinatorial multivalued map on $\cX$ and checks whether the specific constraints are satisfied.
Since \texttt{CMGDB} computes with field coefficients, it returns the Conley index in terms of the reduced rational canonical forms as indicated in \eqref{eq:CIpoly}.
The Hasse diagram of a Morse graph decorated with Conley indices is called a \emph{Conley-Morse graph}.

\subsection{Mathematical interpretation of \texttt{CMGDB} computations}

\begin{defn}
  \label{defn:outerApproximation}
  Let $g\colon Z\to Z$ be a continuous function.
  A combinatorial multivalued map $\cG\colon \cXtop \mvmap \cXtop$ is an \emph{outer approximation} of $g$ if for every $\xi\in \cXtop$
  \begin{equation}
    \label{eqn:outerApproximation}
    g(|\xi|) \subset \Int(|\cG(\xi)|).
  \end{equation}
  In this case, $g$ is called a \emph{selector} of $\cG$.
\end{defn}
\begin{rem}
  \label{rem:trajectory}
  Let $\cG$ be an outer approximation of $g$ and let $z\in B$. Then $z\in |\xi|$ for some $\xi\in\cXtop$, and $g(z)\in |\xi'|$ for some $\xi'\in\cG(\xi)$.
\end{rem}

The following proposition allows us to transition from information about $\cG$ to information about $g$.

\begin{prop}
  \label{prop:attractingblock}
  \cite{kalies:mischaikow:vandervorst:15}
  Assume $\cG$ is an outer approximation for $g$.
  If $\cN\in \sInvset^+(\cG)$, then $|\cN| \in \sABlock(g)$.
\end{prop}

\begin{defn}
  Assume $\cG$ is an outer approximation for $g$.
  Let $\cA \in \sAtt(\cG)$ and let $A = \omega(|\cA|,g)\in\sAtt(g)$.
  The \emph{maximal region of attraction of $A$ under $\cG$} is the maximal element  $\cN\in \sInvset^+(\cG)$ such that $\omega(\cN,\cG) = \cA$.
\end{defn}

We leave the proof of the following result to the reader.

\begin{prop}
  \label{prop:maximal_roa}
  Assume $\cG$ is an outer approximation for $g$ and let  $\cN\in \sInvset^+(\cG)$ be the maximal region of attraction of $A$ under $\cG$.
  If $z \in |\cN|$, then $\omega(z,g)\subset A$.
\end{prop}

\begin{defn}
  Let $\cG$ be an outer approximation for $g$.
  The \emph{Morse decomposition} of $g$ given by $\cG$ is the poset with elements
  \[
    \sMD(g;\cG) := \setdef{\Inv(|\pi^{-1}_\cG(\cM)|,g)}{\cM\in \sMG(\cG)}
  \]
  and partial order $\leq_{\sMG(\cG)}$.
\end{defn}

As discussed in \cite{kalies:mischaikow:vandervorst:21}, if $\Inv(|\pi^{-1}_\cG(\cM)|,g) \neq \emptyset$ for all $\cM\in \sMG(\cG)$, then $\sMD(g;\cG)$ is a Morse representation of $g$. However, $\cG$ is a combinatorial approximation of $g$, and it is possible that $\Inv(|\pi^{-1}_\cG(\cM)|,g) = \emptyset$ for some $\cM\in \sMG(\cG)$. By Proposition~\ref{prop:attractingblock}, $|\eta(\cM)|,|\eta(\cM)^<| \in \sABlock(g)$. Thus the Conley index of $\Inv(|\pi^{-1}_\cG(\cM)|,g)$ is determined by the shift-equivalence class of
\[
  g_*\colon H_*\bigl(|\eta(\cM)|,|\eta(\cM)^<|\bigr)
  \to H_*\bigl(|\eta(\cM)|,|\eta(\cM)^<|\bigr),
\]
which is equivalent to the shift-equivalence class of $\Con_*(\cM)$ determined by \eqref{eq:CIG}.
Applying Proposition~\ref{prop:CInot0} we obtain the following result.

\begin{prop}
  \label{prop:morse_representation}
  Let $\cG$ be an outer approximation for $g$ that gives rise to a Morse decomposition $\sMD(g;\cG)$ and partial order $\leq_{\sMG(\cG)}$.
  If $\Con_*(\cM)$ is nontrivial for all $\cM\in \sMG(\cG)$, then $\sMD(g;\cG)$ is a Morse representation for $g$.
\end{prop}

In practice it often happens that $\Con_*(\cM)$ is trivial for some $\cM\in \sMG(\cG)$.
By Remark~\ref{rem:trivialCon}, we cannot immediately determine whether $\Inv(|\cM|,g)$ is empty.
To investigate this case further \texttt{CMGDB} allows for the option of further cubical subdivisions and a refinement of $\cG$ restricted to $\cM$.
If the recurrent component disappears under refinement, then we can conclude that $\Inv(|\cM|,g)= \emptyset$. This gives rise to the following corollary about
\[
  \sMD'(g;\cG) := \setdef{\Inv(|\pi^{-1}_\cG(\cM)|,g)}{\cM\in \sMG(\cG),\ \Con_*(\cM)\neq 0}.
\]

\begin{cor}
  \label{cor:morse_representation}
  Let $\cG$ be an outer approximation for $g$ that gives rise to a Morse decomposition $\sMD(g;\cG)$.  Assume that for all $\cM\in \sMG(\cG)$, either $\Con_*(\cM)$ is nontrivial or $\Inv(|\pi^{-1}_\cG(\cM)|,g) =\emptyset$. Then
  \[
    \sMD'(g;\cG) \subset \sMD(g;\cG)
  \]
  is a Morse representation for $g$.
\end{cor}

We conclude this section with a straightforward extension of Theorem~\ref{thm:main_simple}.

\begin{cor}
  \label{cor:main}
  Let $(X,d_X)$ and $(Z,d_Z)$ be compact metric spaces. Suppose that $f\colon X \to X$, $g\colon Z \to Z$, and $E\colon X \to Z$ are continuous. Let $\sN_g$ be a finite sublattice of $\sABlock(g)$. If every $N\in\sN_g$ is contained in $E(X)$ and satisfies \eqref{eq:loc_semiconjugacy}, then $\sN_f:=\setdef{E^{-1}(N)}{N\in\sN_g}$ is a finite sublattice of $\sABlock(f)$ and $E^{-1}\colon\sN_g\to\sN_f$ is a lattice isomorphism.
\end{cor}

\begin{proof}
  Preimages preserve finite unions and intersections, and Theorem~\ref{thm:main_simple} shows that every element of $\sN_f$ is an attracting block for $f$. Since the elements of $\sN_g$ lie in $E(X)$, the inverse-image map is injective on $\sN_g$.
\end{proof}

\section{Examples}
\label{sec:applications}

The goal of this section is to demonstrate how our approach can be used in practice and to suggest the applicability of our method to high-dimensional systems.
The details of the computations are contained in Appendix~\ref{sec:appendix}.
In keeping with the philosophy of this paper, we refrain from optimizing the architecture and hyperparameters of the neural networks used to identify the autoencoder and latent dynamics and just report our choices  in Section~\ref{sec:appendix_data_training}.
Sizes of the training  sets are deliberately kept small to highlight the potential of our method for applications where the available data is limited.

In Section~\ref{sec:2D_Leslie} we illustrate the method on a $10$-dimensional system where we know that the global attractor is contained in a $2$-dimensional subspace, and hence we have a correct baseline against which to compare our computations. In particular, the dynamics on the 2-dimensional subspace are given by an overcompensatory Leslie population model.
Extensive work by Ugarcovici and Weiss \cite{ugarcovici:weiss} shows that the dynamics of this model can be complex and extremely sensitive to parameters.
Their work suggests that this model exhibits the Newhouse phenomenon \cite{newhouse,palis:takens}, and thus even with explicit knowledge of the map, it is impossible to rigorously characterize the dynamics at individual parameter values via invariant sets.

In Section~\ref{sec:3d_leslie}, we compare \texttt{CMGDB} computations for a 3-dimensional overcompensatory Leslie model and a learned latent map. On a grid with box widths that are much smaller than the sampled residual estimates, the latent computation yields an additional minimal node that does not satisfy \eqref{eq:loc_semiconjugacy}, so our main theorem provides no corresponding attractor of $f$. We show that the mismatch can be resolved in two ways, either via coarsening the Conley-Morse graph or coarsening the grid.

In Section~\ref{sec:red_coral} we consider a $13$-dimensional model for the population of Mediterranean red coral.
To the best of our knowledge there are no rigorous results on the global dynamics of this system, but numerical studies have identified a one-dimensional attractor that exhibits bistability.
We recover this qualitative information using learned dynamics on a one-dimensional latent space.

Finally, in Section~\ref{sec:chafee_infante} we consider a partial differential equation for which it is known that the global attractor is five-dimensional, the dynamics on the global attractor is known, and the dynamics exhibits bistability.
Using a $64$-dimensional spectral method, we simulate the partial differential equation and obtain latent dynamics in dimensions one, two, and three.
Since the dimension of the latent dynamics is less than the dimension of the attractor, it is impossible to faithfully recover the original dynamics.
However, we show that we can identify bistability and use the latent system to predict the final state with high precision.

\subsection{Extended two-dimensional Leslie model}
\label{sec:2D_Leslie}

To test our method we consider a two-dimensional overcompensatory Leslie model
\begin{equation}
  \label{eq:2dLeslie_cont}
  f_{2,\theta}(x):=
  \begin{bmatrix}
    (\theta_1 x_1+\theta_2 x_2)e^{-0.1(x_1+x_2)}\\
    0.7x_1
  \end{bmatrix}
\end{equation}
and extend it to ten dimensions by adding eight contracting coordinates,
\begin{equation}
  \label{eq:Leslie_cont}
  f_{10,\theta}(x):=
  \begin{bmatrix}
    (\theta_1 x_1+\theta_2 x_2)e^{-0.1(x_1+x_2)}\\
    0.7x_1\\
    0.25x_3\\
    0.25x_4\\
    \vdots\\
    0.25x_{10}
  \end{bmatrix}.
\end{equation}
Observe that the global attractor of \eqref{eq:Leslie_cont} lies on the subspace $[0,\infty)^2 \times \setof{0}^8 \subset [0,\infty)^{10}$, and on this subspace is the same as the global attractor of \eqref{eq:2dLeslie_cont}.
Thus we can compare the Morse graph obtained by applying \texttt{CMGDB} to \eqref{eq:2dLeslie_cont} against the results provided by applying \texttt{CMGDB} to a latent model $g$ for \eqref{eq:Leslie_cont}.

For both models we set $\theta=(23.5,23.5)$.
Figures~\ref{fig:lesliecontraction_dynamics}(a)--(b) show the result of applying \texttt{CMGDB} to \eqref{eq:2dLeslie_cont} on the domain $[0,90]\times[0,70]$.
More specifically, Figure~\ref{fig:lesliecontraction_dynamics}(a) shows the Hasse diagram of $\sMG(\cF)$. Written on each node is a label of the form $k: (q_0(x), q_1(x), q_2(x))$ where $k$ identifies the poset element and $(q_0(x), q_1(x), q_2(x))$ denotes the Conley index.
Figure~\ref{fig:lesliecontraction_dynamics}(b) shows a plot of the phase space where for all $\cM \in \sMG(\cF)$, the region $|\pi^{-1}_\cF(\cM)|$ is plotted with a color corresponding to the node in the Hasse diagram.
Node 4 of $\sMG(\cF)$ corresponds to the unstable fixed point at the origin. Since the index of node 4 is trivial, the computation is inconclusive on whether the corresponding invariant set is empty or not. Therefore the poset of invariant sets $\sMD'(f_{2,\theta};\cF)$ does not contain the fixed point and does not satisfy the assumptions of Corollary~\ref{cor:morse_representation}. Thus although $\sMD'(f_{2,\theta};\cF)$ is not a Morse representation, it is still a useful object because it captures the order relation from $\sMG(\cF)$ on Morse sets whose existence can be inferred from the Conley index.

To learn the encoder and latent map $g\colon Z \to Z$, $Z\subset \R^2$, we set $X=[0,90]\times[0,70]\times[0,100]^8$ and uniformly sampled 8,000 initial points for training and 2,000 for validation, and considered $20$ iterations of each point.
The training set was formed by the points $(f^k_{10,\theta}(x_i),f^{k+1}_{10,\theta}(x_i))$ for $0\leq k\leq 19$.
Figure~\ref{fig:lesliecontraction_dynamics}(c) shows the Morse graph $\sMG(\cG)$ constructed from $g$.

\begin{figure}[htbp]
  \centering
  \begin{subfigure}[b]{0.45\textwidth}
    \centering
    \includegraphics[width=\textwidth]{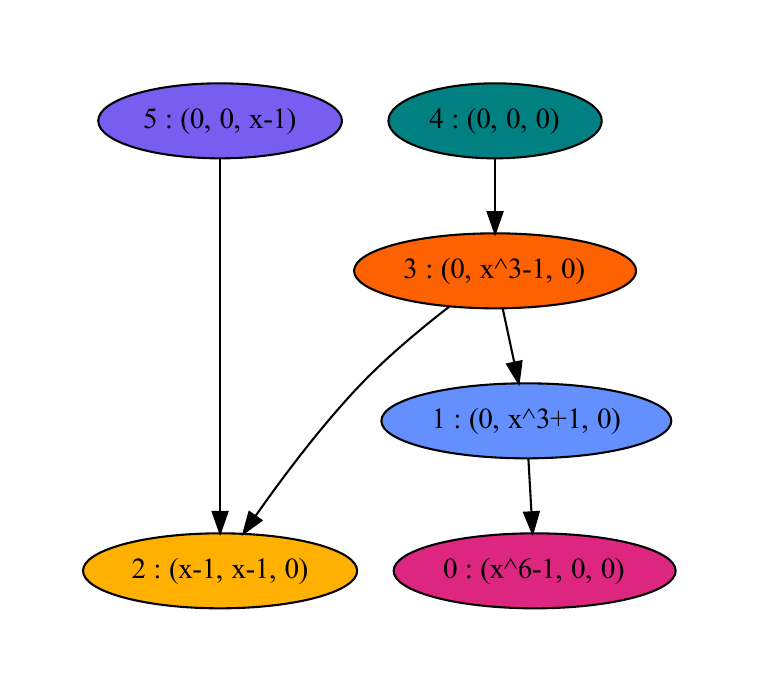}
    \caption{}
  \end{subfigure}
  \hfill
  \begin{subfigure}[b]{0.45\textwidth}
    \centering
    \includegraphics[width=\textwidth]{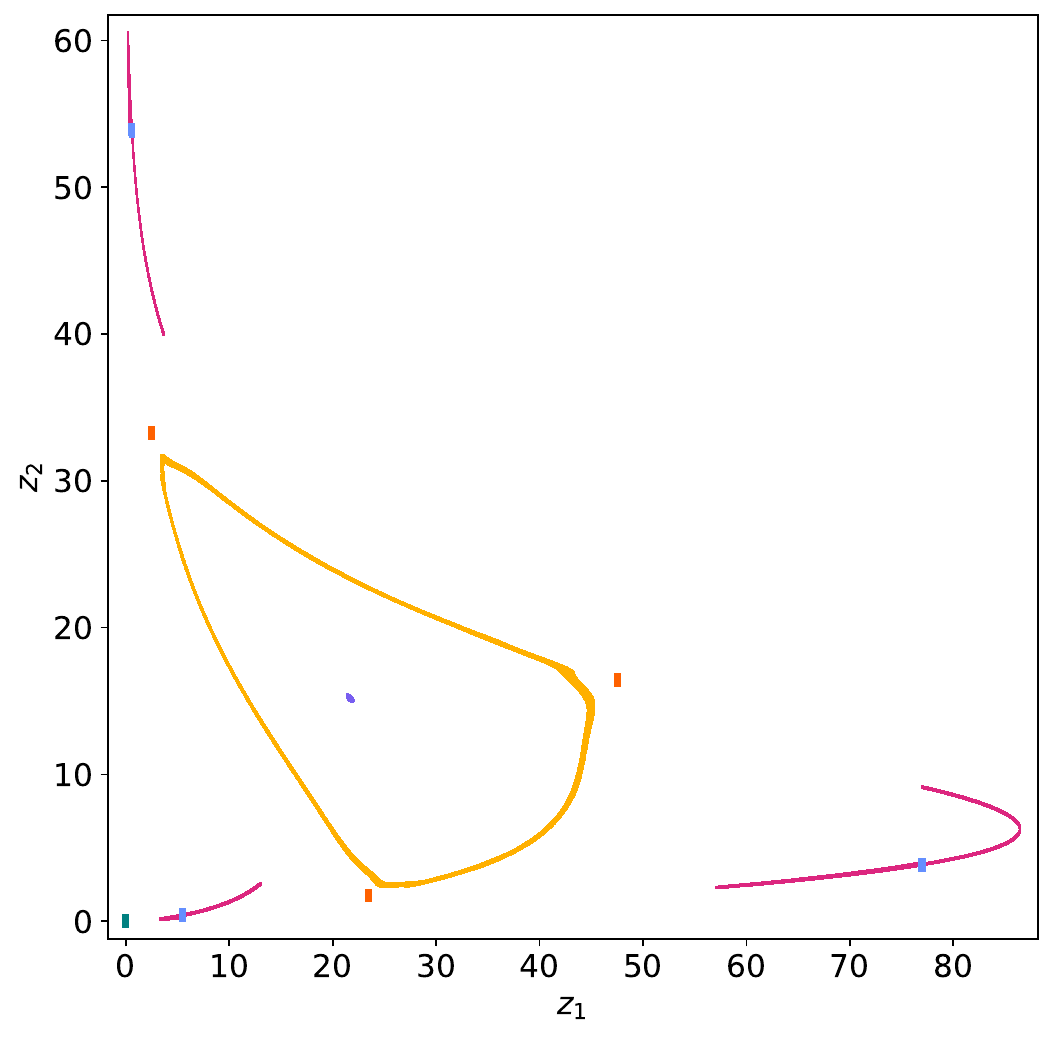}
    \caption{}
  \end{subfigure}
  \par\medskip
  \begin{subfigure}[b]{0.45\textwidth}
    \centering
    \includegraphics[width=\textwidth]{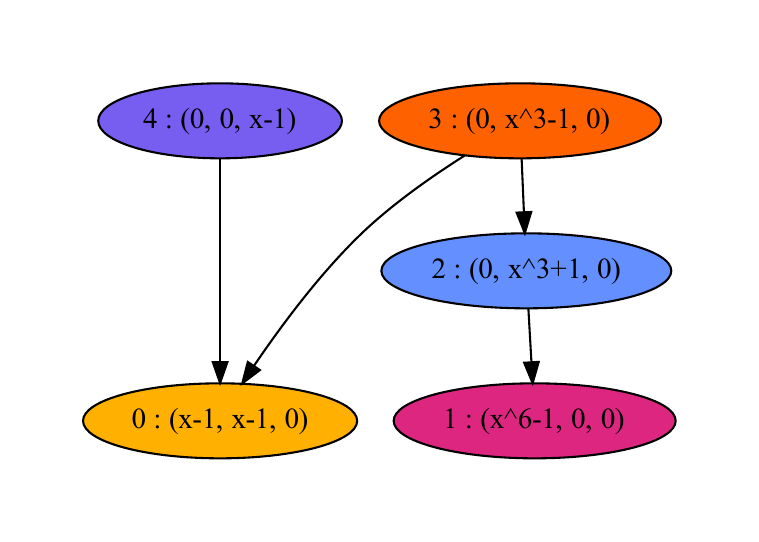}
    \caption{}
  \end{subfigure}
  \hfill
  \begin{subfigure}[b]{0.45\textwidth}
    \centering
    \includegraphics[width=\textwidth]{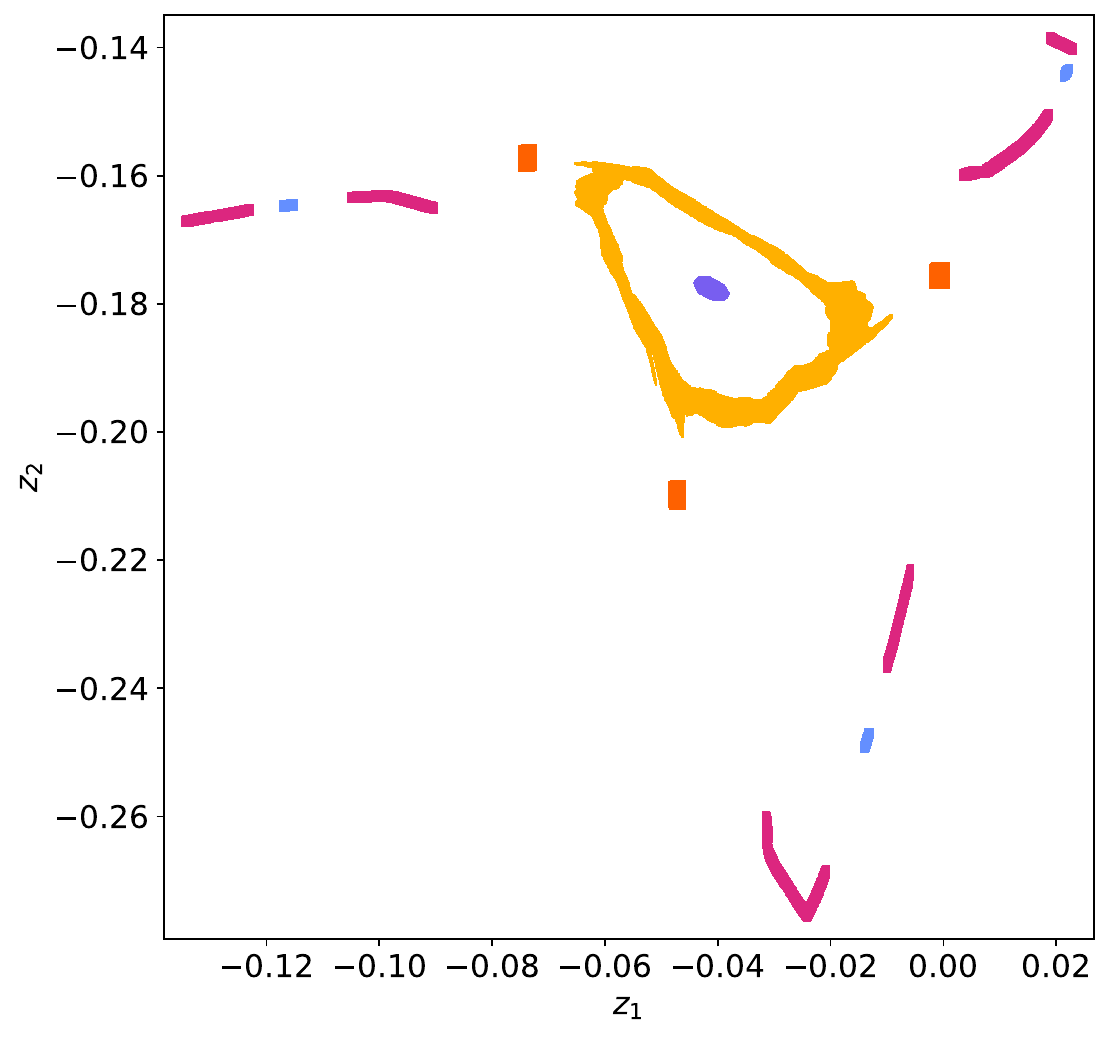}
    \caption{}
  \end{subfigure}
  \caption{
    (a) Hasse diagram of $\sMG(\cF)$, the baseline Morse graph for \eqref{eq:2dLeslie_cont} computed directly from the original two-dimensional subsystem.
    (b) Corresponding color-coded regions of phase space in the original $(x_1,x_2)$ coordinates.
    (c) Hasse diagram of $\sMG(\cG)$, the Morse graph obtained from the latent model for \eqref{eq:Leslie_cont}.
    (d) Corresponding color-coded regions of the latent phase space.
  For visualization purposes, in plot (b) the set $|\pi^{-1}_\cF(k)|$ is enlarged by a factor of $20$ for $k = 0, 2, 5$ and by a factor of $100$ for $k = 1, 3, 4$. In plot (d), $|\pi^{-1}_\cG(k)|$ is enlarged by a factor of $20$ for $k = 1, 2$ and by a factor of $50$ for $k = 3$.}
  \label{fig:lesliecontraction_dynamics}
\end{figure}

There is an isomorphism between $\sMG(\cG)$ and $\sMD'(f_{2,\theta};\cF)$.
It is not surprising, based on the scarcity of training data near the origin, that $\cG$ does not identify any recurrence associated with the unstable fixed point at the origin (see \cite{gameiro:gelb:mischaikow} for a discussion on this point).

Moreover, the geometries in Figures~\ref{fig:lesliecontraction_dynamics}(b) and (d) cannot be directly compared,
which indicates that the encoder $E$ is not the standard projection from $\R^{10}$ to $\R^2$ onto the first two components.
It is therefore noteworthy, given the relative paucity of data in $\R^{10}$, that using $g$ it is possible to correctly identify the dynamics up to homology of the invariant sets.

To be more specific, the Conley index of a stable invariant circle is $(x-1,x-1,0)$.
This structure is identified via the latent dynamics.
The index $(x^6-1,0,0)$ identifies a six-component attracting set. The indices $(0,x^3-1,0)$ and $(0,x^3+1,0)$ agree with orientation-preserving and orientation-reversing unstable period-three orbits, respectively.
These structures are identified in the latent computation even though, as shown in Table~\ref{tab:sampled_residual_tolerance}, sampled residuals violate \eqref{eq:loc_semiconjugacy} on the candidate attracting blocks. This agreement suggests that the sufficient condition is conservative.

\subsection{Three-dimensional Leslie model}
\label{sec:3d_leslie}

To examine the interaction between model error and combinatorial resolution, we consider the three-dimensional overcompensatory Leslie population model
\begin{equation}
  \label{eq:Leslie3D}
  f_\theta(x):=
  \begin{bmatrix}
    (\theta_1 x_1+\theta_2 x_2+\theta_3x_3)e^{-0.1(x_1+x_2+x_3)}\\
    0.7x_1\\
    0.7x_2
  \end{bmatrix}
\end{equation}
at the parameter value $\theta=(28.9,29.8,22.0)$ and $X=[0,220]\times[0,154]\times[0,108]$.

For the reference computation, we restrict to $B=[0,110]\times[0,77]\times[0,54]$. The estimates $f_\theta^3(X)\subset B$ and $f_\theta(B)\subset B$ show that every invariant set in $X$ lies in $B$. Using \texttt{CMGDB} we compute a combinatorial multivalued map
$\cF$ on a discretization of $B$. The Conley-Morse graph $\sMG(\cF)$, shown in Figure~\ref{fig:3D_Leslie_direct}, has exactly two minimal nodes, both with index $(x^4-1,0,0,0)$. Stable period-four orbits in the corresponding components are approximately
\[
  \begin{aligned}
    P_0\approx{}&\{(102.59,4.63,0.59),(0.065,71.82,3.24),(1.21,0.045,50.27),(6.61,0.85,0.032)\},\\
    P_1\approx{}&\{(20.09,2.26,21.11),(14.41,14.06,1.58),(43.08,10.09,9.84),(3.23,30.16,7.06)\}.
  \end{aligned}
\]
This establishes the reference against which we compare the latent computations.
We emphasize that $\sMD'(f; \cF)$ captures the dynamics of $f$ at a certain level of resolution, and provides a lower bound for the possible dynamics. As we will see, it is possible for the latent computations to capture more or less fine dynamics while still being consistent with $\sMD'(f; \cF)$.

\begin{figure}[!htbp]
  \centering
  \begin{subfigure}[c]{0.49\textwidth}
    \centering
    \includegraphics[width=\textwidth]{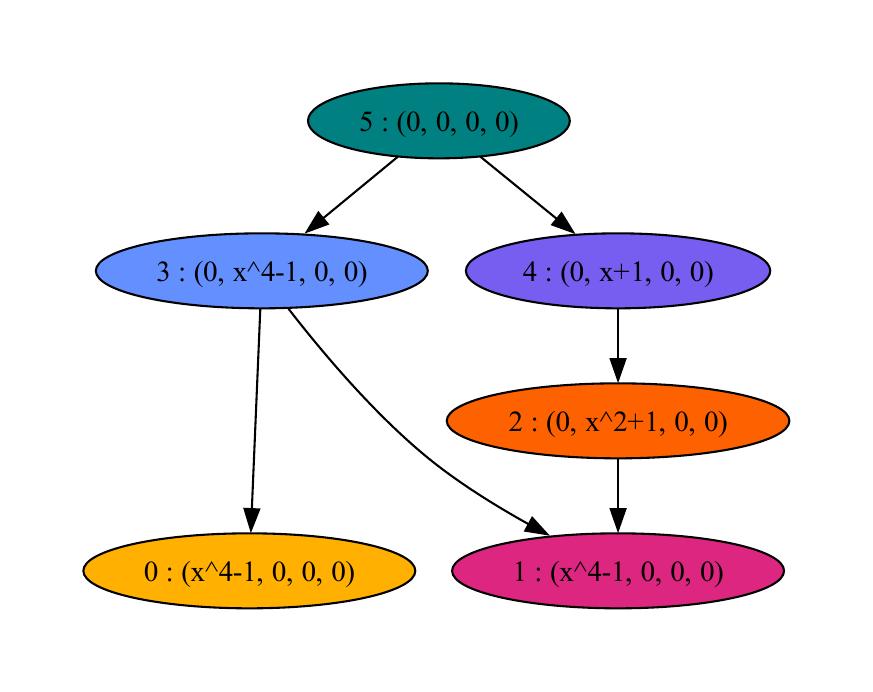}
    \caption{}
  \end{subfigure}
  \hfill
  \begin{subfigure}[c]{0.49\textwidth}
    \centering
    \includegraphics[width=\textwidth]{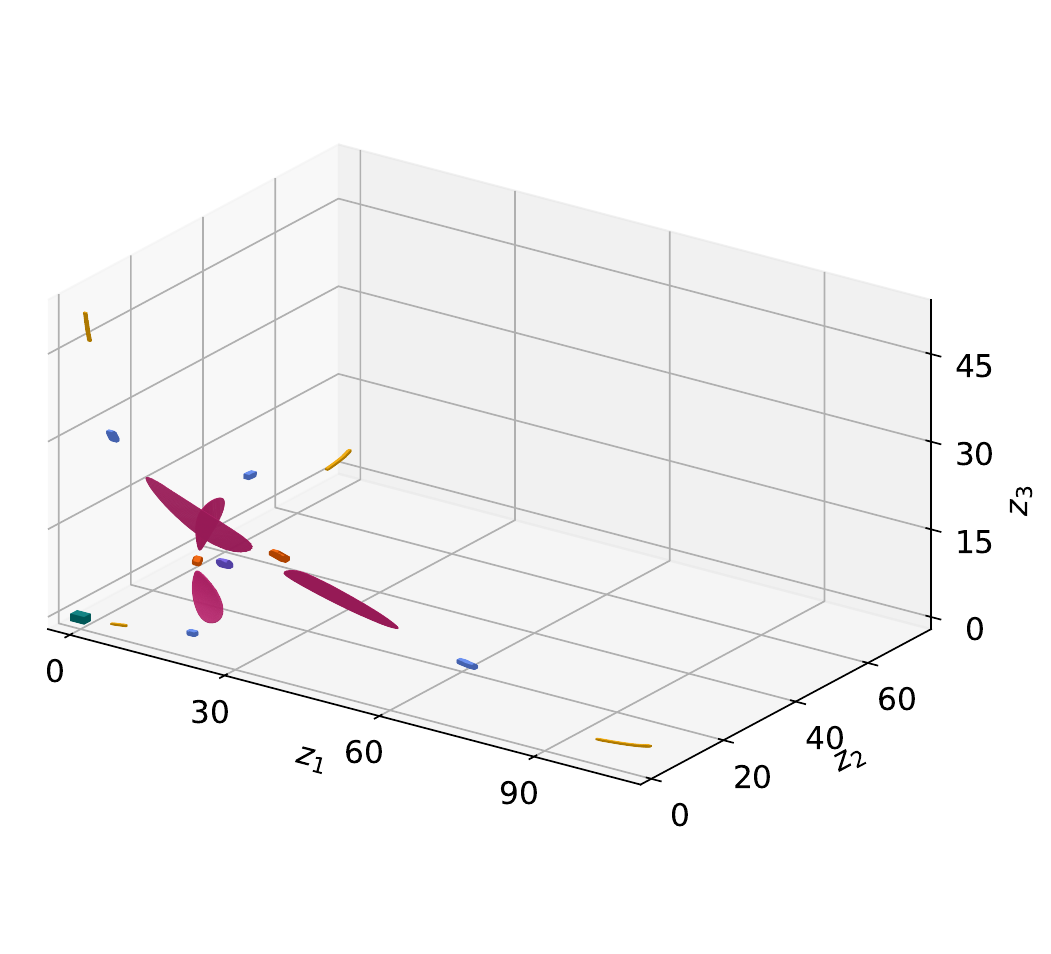}
    \caption{}
  \end{subfigure}
  \caption{
  (a) Hasse diagram of $\sMG(\cF)$, the baseline Morse graph for \eqref{eq:Leslie3D} computed directly from the original system. (b) Corresponding color-coded regions of phase space. The sets $|\pi^{-1}_\cF(0)|$ and $|\pi^{-1}_\cF(1)|$ contain the stable period-four orbits $P_0$ and $P_1$. For visualization purposes, $|\pi^{-1}_\cF(k)|$ is enlarged by a factor of $10$ for $k = 0$, by a factor of $20$ for $k = 2, 3, 4$, and by a factor of $50$ for $k = 5$.}
  \label{fig:3D_Leslie_direct}
\end{figure}

We obtain an encoder and a two-dimensional latent map $g_1$ using 4,000 uniformly sampled initial conditions, split into 3,200 training and $800$ validation trajectories. After discarding the first ten iterates, the next twenty iterates are retained.

The resulting Conley-Morse graph $\sMG(\cG_1)$ in Figure~\ref{fig:3D_Leslie_latent_fine_MG} has three minimal nodes, labeled $0$, $1$, and $4$. Nodes $0$ and $1$ have the same attracting period-four indices observed in the reference computation, whereas node $4$ indicates the existence of an attractor for $g_1$ that does not correspond to an attractor in $\sMD'(f;\cF)$. In Table~\ref{tab:sampled_residual_tolerance} we report sampled residual estimates.
The estimates for node $4$ violate \eqref{eq:loc_semiconjugacy}, so Theorem~\ref{thm:main_simple} does not resolve whether $f$
exhibits tristability.

\begin{figure}[!htbp]
  \centering
  \begin{subfigure}[c]{0.45\textwidth}
    \centering
    \includegraphics[width=\textwidth]{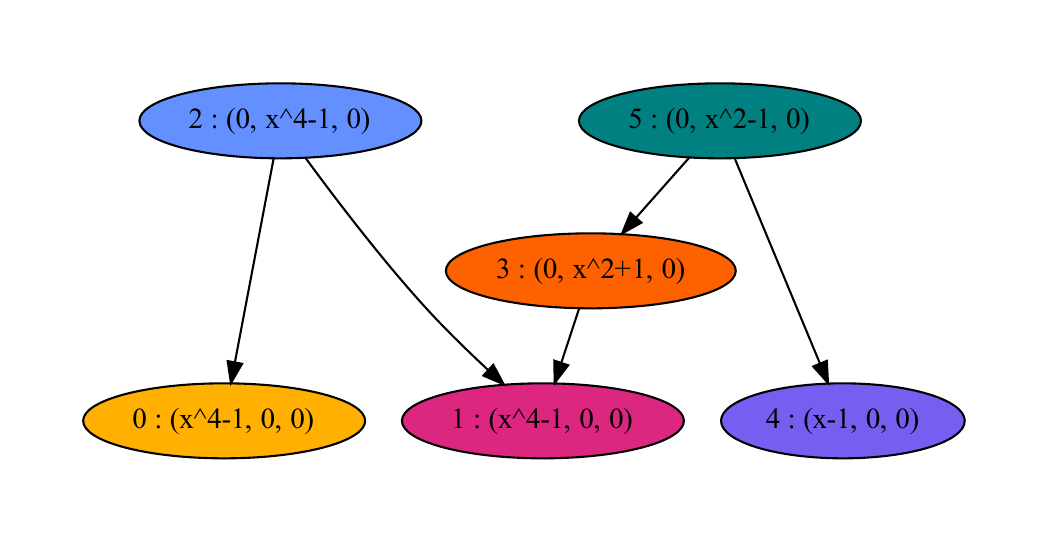}
    \caption{}
    \label{fig:3D_Leslie_latent_fine_MG}
  \end{subfigure}
  \begin{subfigure}[c]{0.45\textwidth}
    \centering
    \includegraphics[width=\textwidth]{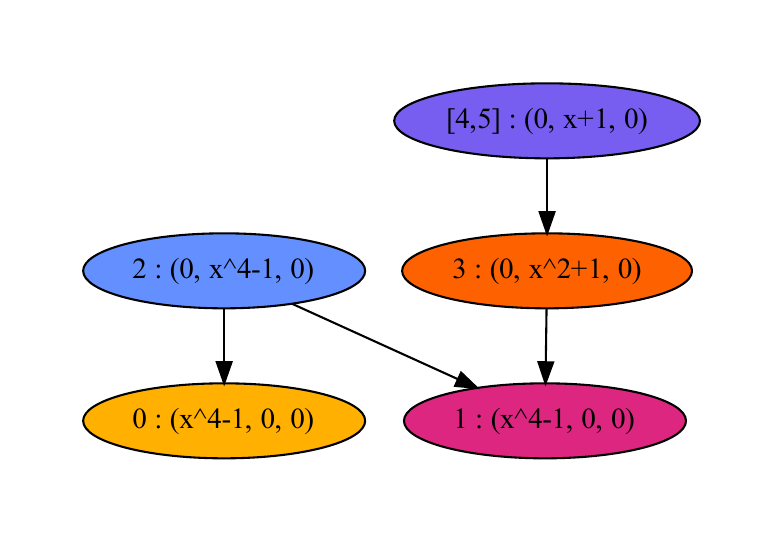}
    \caption{}
    \label{fig:3D_Leslie_latent_merged_morsegraph}
  \end{subfigure}
  \begin{subfigure}[c]{0.43\textwidth}
    \centering
    \includegraphics[width=\textwidth]{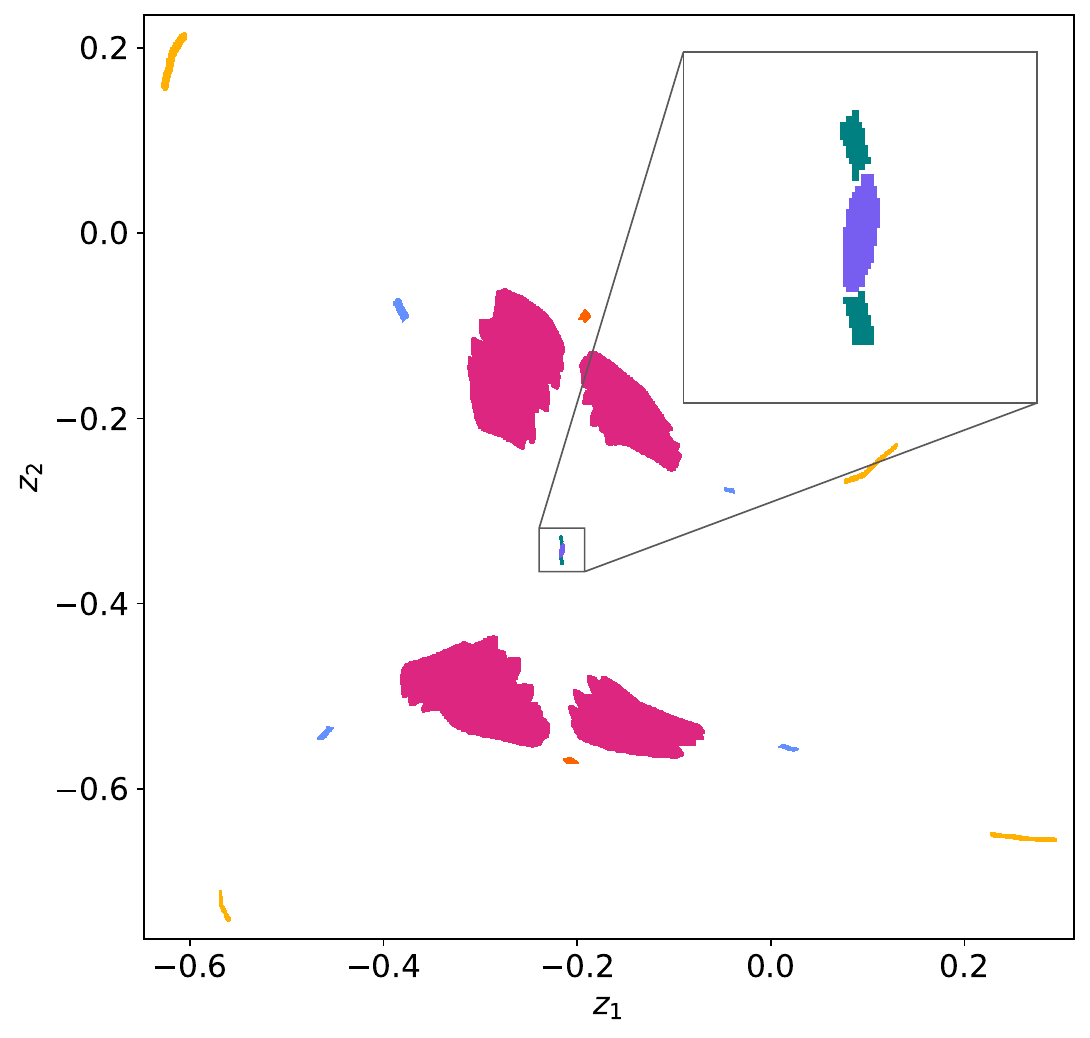}
    \caption{}
    \label{fig:3D_Leslie_latent_fine_morsesets}
  \end{subfigure}\hfill
  \begin{subfigure}[c]{0.43\textwidth}
    \centering
    \includegraphics[width=\textwidth]{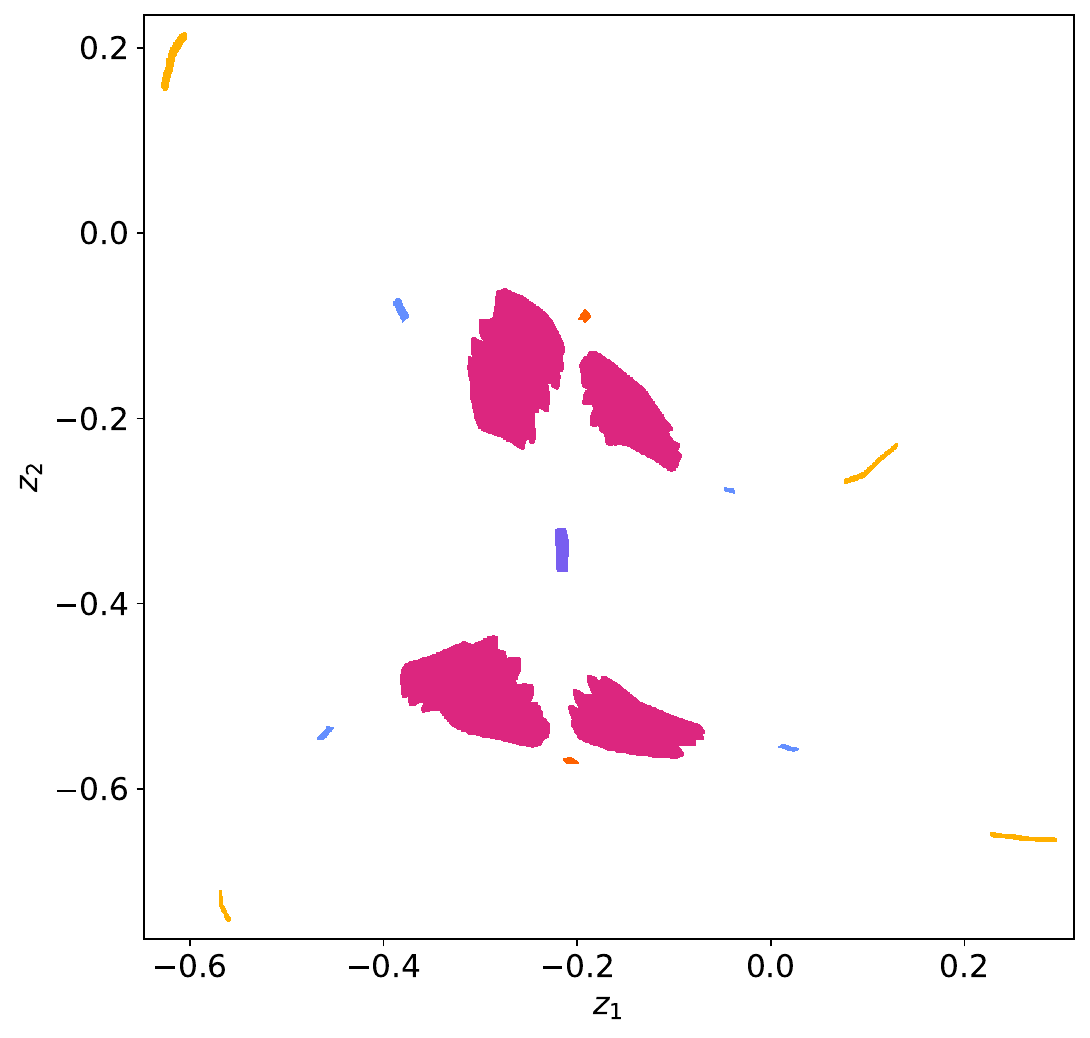}
    \caption{}
    \label{fig:3D_Leslie_latent_merged_morsesets}
  \end{subfigure}
  \caption{
  (a) Hasse diagram of $\sMG(\cG_1)$, where $\cG_1$ is computed from a latent dynamics model of \eqref{eq:Leslie3D}. (b) Hasse diagram of a Morse graph that is obtained by coarsening $\sMG(\cG_1)$. (c, d) Color-coded regions of phase space corresponding to (a) and (b) respectively. In (c) the inset magnifies $|\pi^{-1}_{\cG_1}(4)|$ and $|\pi^{-1}_{\cG_1}(5)|$. In (d) the coarsened region corresponding to Morse node $[4, 5]$ is enlarged by a factor of $20$.}
  \label{fig:3D_Leslie_latent}
\end{figure}

We consider two ways to interpret the results. First, $\sMG(\cG_1)$ identifies dynamics that are finer but consistent with the dynamics expressed by $\sMD'(f;\cF)$. This can be seen by considering a coarsening of $\sMG(\cG_1)$. To obtain the coarsened Morse graph, depicted in Figure~\ref{fig:3D_Leslie_latent_merged_morsegraph}, we identify nodes $4$ and $5$ of $\sMG(\cG_1)$ as a single node $[4,5]$. The region of phase space corresponding to node $[4, 5]$ is defined to be the union of the original regions of phase space for nodes $4$ and $5$ together with the geometric realization of the connecting orbit under $\cG_1$. The resulting set has Conley index $(0,x+1,0)$.

The coarsened Morse graph has two minimal elements and is isomorphic to $\sMD'(f;\cF)$. Moreover, we recover the full Conley index information associated with $\sMG(\cF)$. Using subscripts to distinguish elements of $\sMG(\cG_1)$ and $\sMG(\cF)$, these ideas can be made precise by constructing a poset epimorphism $\nu \colon \sMG(\cG_1) \to \sMD'(f;\cF)$ defined by
\[
  \nu(4_{\sMG(\cG_1)})=\nu(5_{\sMG(\cG_1)})=\Inv(|\pi_{\cF}^{-1}(4_{\sMG(\cF)})|, f).
\]
We see that the finer information yielded by $\sMG(\cG_1)$ is consistent with dynamics expressed by $\sMD'(f; \cF)$, even though the existence of a third attractor for $f$ is not resolved.

As a second approach, we coarsen the cubical grid itself to produce a more coarse, combinatorial multivalued map. This is motivated by the fact that the estimates in Table~\ref{tab:sampled_residual_tolerance} are much larger than the box widths of the grid used to obtain $\sMG(\cG_1)$ (which are approximately $(4.04\times10^{-4},7.83\times10^{-4})$), meaning that $\cG_1$ distinguishes features of $g_1$ at a scale smaller than the sampled discrepancy between $g_1\circ E$ and $E\circ f_\theta$. We produce the new multivalued map $\cG_2$ from $g_1$ on a grid with  boxes of side lengths approximately $(8.08\times10^{-4},7.83\times10^{-4})$.

The Morse graph $\sMG(\cG_2)$ has two minimal nodes. In this case, the sampled tolerances of attracting blocks determined from the two minimal nodes are approximately one box width (see Table~\ref{tab:sampled_residual_tolerance}). The Morse graph obtained by omitting nodes with trivial Conley index is shown in Figure~\ref{fig:3D_Leslie_latent_coarse} and exhibits bistability. The regions associated to nodes $4$ and $5$ of the previous calculation now belong to a single nonminimal recurrent component, so the additional attractor is no longer present at this resolution.

The two minimal nodes have indices $(x^4-1,0,0)$ and $(x^2-1,0,0)$. Thus the coarser computation recovers the two attracting components of the original system, but it does not reproduce the period-four Conley index for both components. This loss of index-level information is compatible with the scope of Theorem~\ref{thm:main_simple}, which provides conditions for lifting attracting blocks and concluding the existence of attractors, while lifting the Conley index remains an open problem.
\begin{figure}[!htbp]
  \centering
  \begin{subfigure}[c]{0.44\textwidth}
    \centering
    \includegraphics[width=\textwidth]{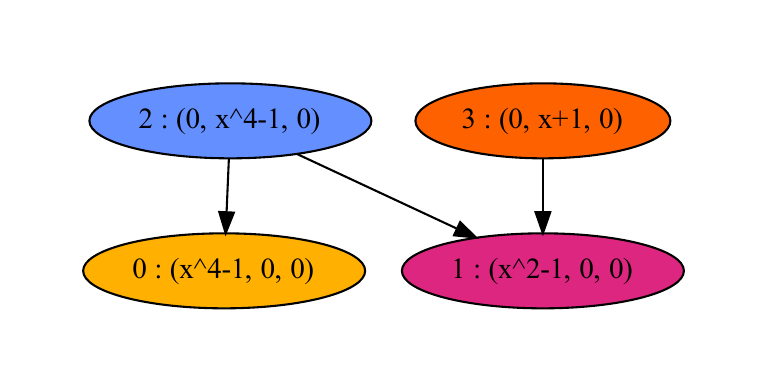}
    \caption{}
  \end{subfigure}
  \hfill
  \begin{subfigure}[c]{0.54\textwidth}
    \centering
    \includegraphics[width=\textwidth]{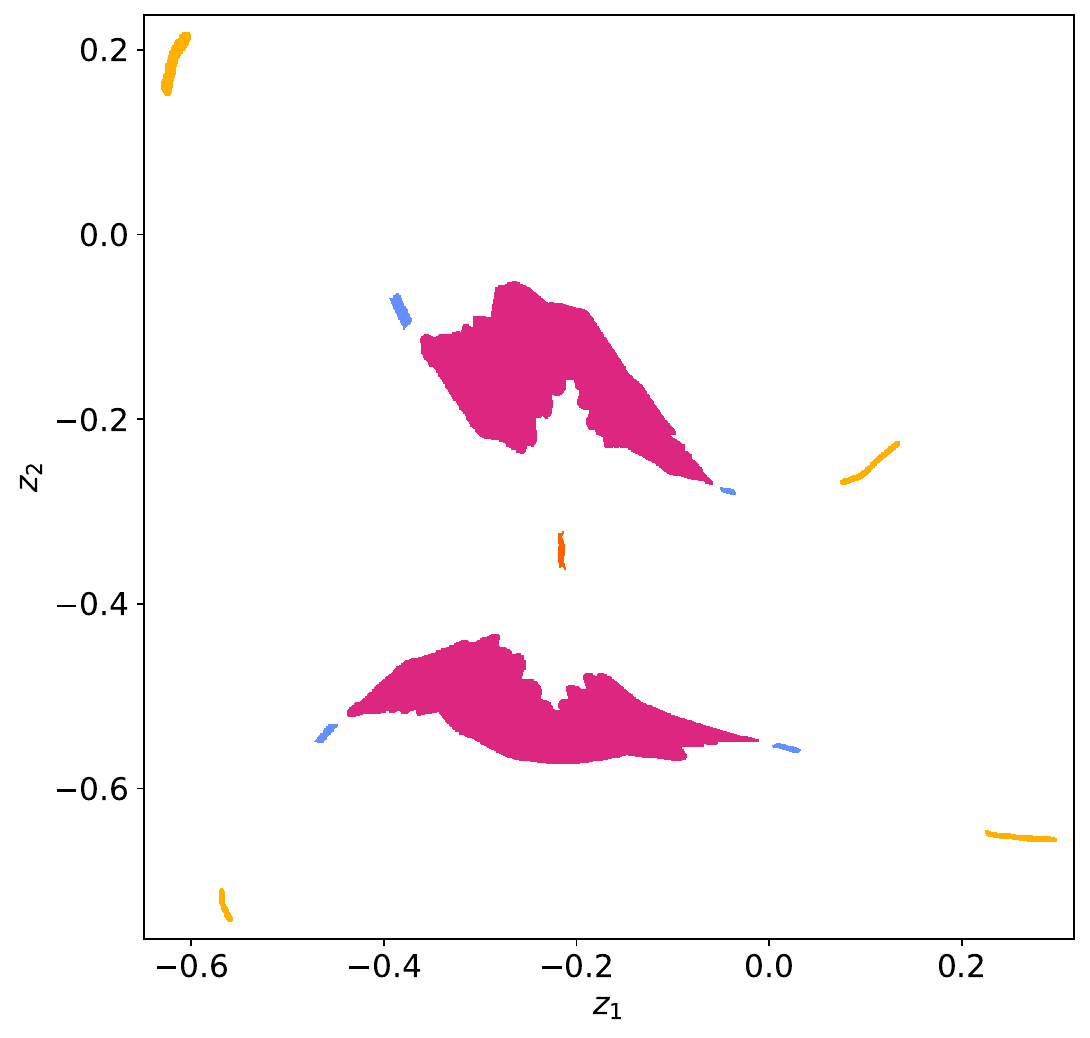}
    \caption{}
  \end{subfigure}
  \caption{
    (a) Hasse diagram of $\sMG(\cG_2)$ where $\cG_2$ is computed at a lower grid resolution than $\cG_1$ using $g_1$, the latent dynamics model of \eqref{eq:Leslie3D}. (b) Corresponding color-coded regions of latent phase space.
  }
  \label{fig:3D_Leslie_latent_coarse}
\end{figure}

To summarize, we have defined two different coarsenings of a computation performed with the latent dynamics model. The first changes the Morse graph obtained at a fine resolution, while the second recomputes the multivalued map on a coarser grid. Both approaches result in two minimal nodes, but in neither case are we able to verify the uniform residual hypothesis of Theorem~\ref{thm:main_simple}.

\subsection{Red coral population model}
\label{sec:red_coral}

We consider a 13-dimensional model for the population of Mediterranean red coral that was proposed in \cite{santangelo:bramanti:iannelli} on the basis of demographic data. The population consists of thirteen age classes binned by year with the size of each class denoted by $x_i$. The system is given by
\begin{equation}
  \label{eq:coral_larval_survival}
  s^{(c_1, c_2, \alpha)}(u) = \frac{c_1}{u+c_2e^{-\alpha u}},
\end{equation}
where $u(x) = \frac{1}{\Omega}\sum_{i=2}^{13} x_i$ and the full model takes the form
\begin{equation}
  \label{eq:coral}
  f(x)_i =
  \begin{cases}
    s^{(c_1, c_2, \alpha)}(u(x)) \displaystyle\sum_{j=1}^{13} b_j x_j & i = 1\\
    s_{i-1} x_{i-1} & 2 \leq i \leq 13.
  \end{cases}
\end{equation}
We use the parameters derived in \cite{santangelo:bramanti:iannelli}, setting $\Omega = 36$, $c_1 = 2.94$, $c_2 = 520$, and $\alpha = 0.14$. The values $s_i$ for $1 \leq i \leq 12$ and $b_i$ for $1 \leq i \leq 13$ are provided in Table~\ref{tab:coral_data} in Section~\ref{sec:appendix_data_training}.

To the best of our knowledge the global dynamics of $f$ has not been rigorously determined, and the dimension of the system prohibits computation of the Morse graph in the full phase space using \texttt{CMGDB}. However, what is known \cite{santangelo:bramanti:iannelli,kamimoto:kim:sander:wanner} is that this system has two distinct attracting fixed points $a_0$, $a_1$ (one of them is the origin) and an unstable fixed point $r$ with connecting orbits to both $a_0$ and $a_1$. This is a lower bound on the complexity of the dynamics and can be encoded via a Morse graph with three nodes: two attractor nodes, each corresponding to one of the known stable fixed points,  and one unstable node corresponding to the unstable fixed point. Furthermore, this dynamics can be encoded in a one-dimensional system, which suggests learning a one-dimensional latent dynamical system.

For the computation shown below, we created the training sample from $20$ iterations of $500$ initial conditions obtained by Sobol sampling \cite{Bratley1988Algorithm, sobol} and obtained an encoder $E$ and latent dynamics $g\colon \R \to \R$. The initial conditions were sampled from $\prod_{i=1}^{13}[0,u_i]$, where
\[
  (u_{1},\ldots,u_{13})=(1300,1150,750,520,270,120,35,20,7,5,5,2,2).
\]
The resulting Morse graph $\sMG(\cG)$ is shown in Figure~\ref{fig:coral_latent_dynamics}(a). As desired, the Conley-Morse graph matches the known dynamics. There are two minimal nodes $0$ and $1$, each with the Conley index of a fixed point. Thus, we recover bistability and the nature of the attractors up to homology. Node $2$ has the Conley index of an unstable fixed point and the order relation suggests the existence of connecting orbits from $|\pi^{-1}_\cG(2)|$ to $|\pi^{-1}_\cG(0)|$ and $|\pi^{-1}_\cG(1)|$. In summary, our method provides us with a qualitative understanding of the dynamics of \eqref{eq:coral} that matches results from classical numerical studies. As shown in Table~\ref{tab:sampled_residual_tolerance}, the sampled residual violates \eqref{eq:loc_semiconjugacy} on both candidate attracting blocks, so this agreement is not certified by Theorem~\ref{thm:main_simple}.

\begin{figure}[!htpb]
  \centering
  \begin{subfigure}[b]{0.35\textwidth}
    \centering
    \includegraphics[width=\textwidth]{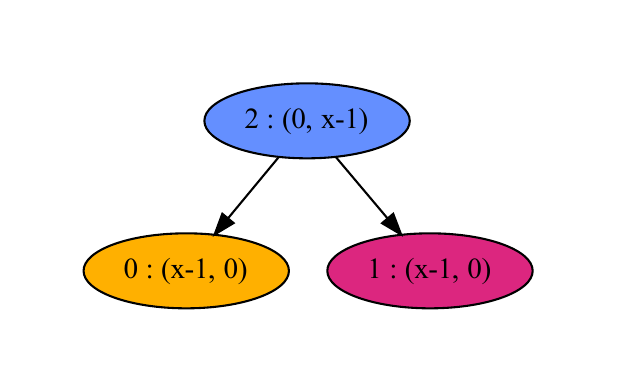}
    \caption{}
    \label{fig:coral_latent_dynamics_mg}
  \end{subfigure}
  \begin{subfigure}[b]{0.64\textwidth}
    \centering
    \includegraphics[width=\textwidth]{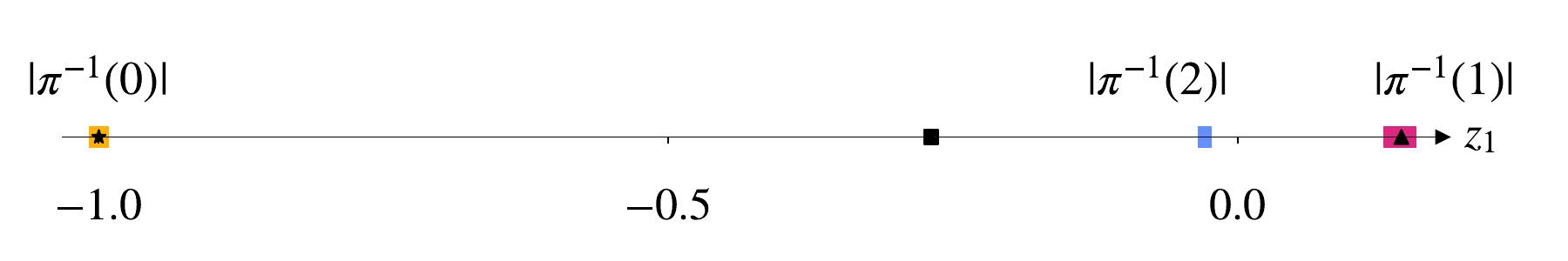}
    \caption{}
    \label{fig:coral_latent_dynamics_morse_sets}
  \end{subfigure}
  \caption{
  (a) Hasse diagram of $\sMG(\cG)$ for a one-dimensional latent dynamics model for \eqref{eq:coral}.  (b) Corresponding color-coded regions of phase space with images of the known fixed points: $E(a_0)$, $E(a_1)$, and $E(r)$, indicated by a star, a triangle, and a square, respectively.}
  \label{fig:coral_latent_dynamics}
\end{figure}

We can also ask about a quantitative agreement, i.e., does $E$ identify   attractors of $f$ with appropriate attractors of $g$? Figure~\ref{fig:coral_latent_dynamics}(b) shows the encoded images $E(a_0)$,  $E(a_1)$ and $E(r)$ as a star, triangle, and square, respectively. Observe that the stable fixed points are mapped to two distinct attracting blocks $|\pi_\cG^{-1}(0)|$ and $|\pi_\cG^{-1}(1)|$ corresponding to the two minimal nodes in the Morse graph. However, the image of the unstable fixed point $E(r)$ does not belong to the isolating neighborhood $|\pi_\cG^{-1}(2)|$ of the unstable node, but rather is mapped to an attracting block for node $0$.

The fact that $E(a_0)\in |\pi_\cG^{-1}(0)|$, $E(a_1)\in |\pi_\cG^{-1}(1)|$, but $E(r)\notin|\pi_\cG^{-1}(2)|$ is not completely surprising. Our training set consists of $10^4$ state pairs from a 13-dimensional state space, and given the nature of the dynamics, the density of the data is much greater around the stable fixed points. This provides a plausible explanation for the quantitative agreement with respect to $a_0$ and $a_1$, but not $r$.
We return to this topic in Section~\ref{sec:conclusions}.

\subsection{Chafee-Infante}
\label{sec:chafee_infante}

The Chafee-Infante equation \cite{chafee:infante}, which is a specific form of the Allen-Cahn equation \cite{allen:cahn}, is
\begin{equation}
  \label{eq:chafee_infante}
  \begin{cases}
    u_t = u_{xx} + \lambda (u - u^3), \quad x \in [0, \pi], t \geq 0 \\
    u(0, t) = u(\pi, t) = 0, \quad t \geq 0 \\
    u(x, 0) = u_0(x), \quad x \in [0, \pi].
  \end{cases}
\end{equation}
This is a variational equation with a global compact attractor. Thus, solutions on the global attractor consist of equilibria, which are solutions of the following boundary value problem
\begin{equation}
  \label{eq:ci_steady_state}
  \begin{cases}
    u_{xx} + \lambda (u - u^3) = 0, \\
    u(0) = u(\pi) = 0,
  \end{cases}
\end{equation}
or connecting orbits between the equilibria. The structure of the global attractor is well understood for all parameter values $\lambda$. In particular, it is known that the complete bifurcation diagram for the equilibria for  $\lambda \in (0,45]$ is as shown in Figure~\ref{fig:ci_bif_diagram} \cite{henry}.

\begin{figure}[!htpb]
  \centering
  \includegraphics[width=0.7\textwidth]{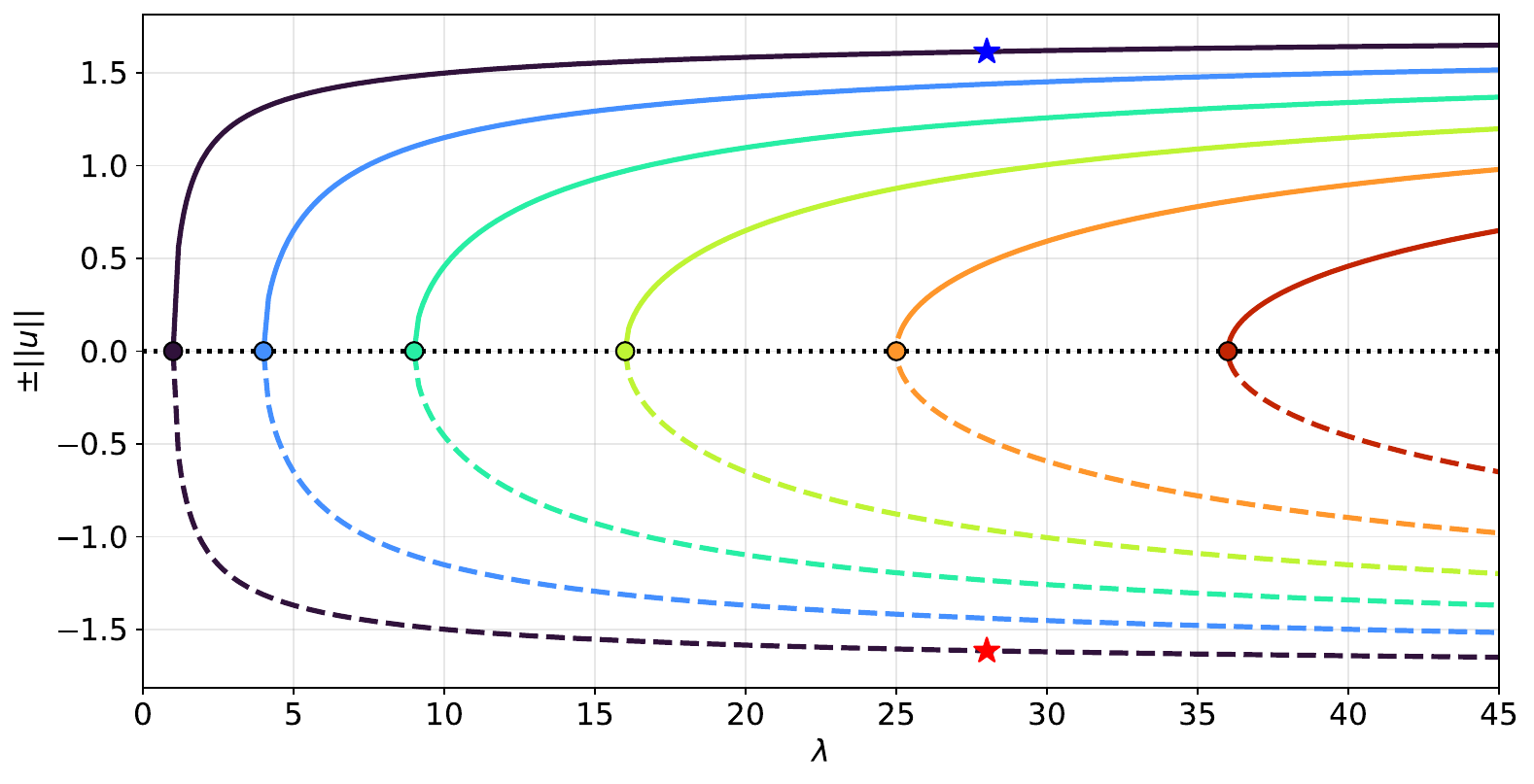}
  \caption{Bifurcation diagram of steady state solutions of \eqref{eq:chafee_infante}. At $\lambda = 28$ there are nine unstable solutions and two stable ones (indicated by the stars in the figure).}
  \label{fig:ci_bif_diagram}
\end{figure}

For this paper we fix the parameter value to be $\lambda = 28$, and as shown in Figure~\ref{fig:ci_bif_diagram} there are $11$ hyperbolic equilibria. The trivial solution $u = 0$, which we denote by $M(5)$, has an unstable manifold of dimension $5$. For $k=0,\ldots, 4$, there are pairs of equilibria with unstable manifold of dimension $k$ that we denote by $M(k^+)$ and $M(k^-)$ depending on the sign of $u'(0)$. Observe that $M(0^\pm)$ are two stable equilibria (indicated by the stars in Figure~\ref{fig:ci_bif_diagram}). The Morse representation with these equilibria as the Morse sets takes the form shown in Figure~\ref{fig:ci_MR} \cite{mischaikow:95}. The Conley indices of the Morse sets are
\[
  \Con_n(M(5)) =
  \begin{cases}
    x-1 & \text{if $n=5$,} \\
    0 & \text{otherwise,}
  \end{cases}
  \quad\text{and}\quad
  \Con_n(M(k^\pm)) =
  \begin{cases}
    x-1 & \text{if $n=k$,} \\
    0 & \text{otherwise,}
  \end{cases}
\]
and the dimension of the global attractor is $5$. This suggests learning a latent model $g\colon \R^5\to \R^5$. However, applying \texttt{CMGDB} to a five-dimensional map is computationally expensive, and thus we focus on understanding the most compelling dynamics of \eqref{eq:chafee_infante}.

\begin{figure}[!htpb]
  \centering
  \begin{subfigure}[b]{0.33\textwidth}
    \centering
    \includegraphics[width=0.72\textwidth]{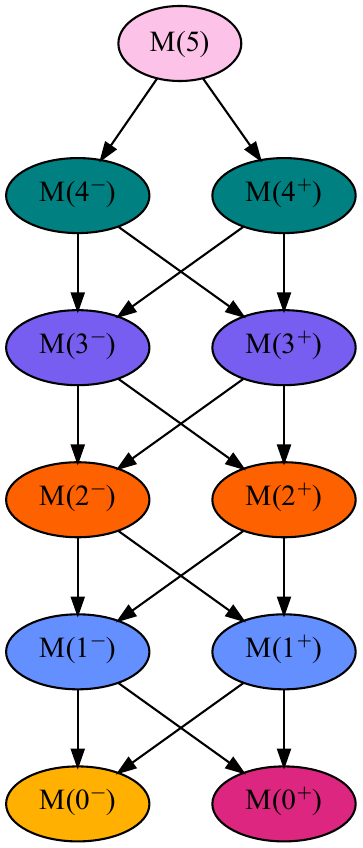}
    \caption{}
    \label{fig:ci_MR}
  \end{subfigure}
  \begin{subfigure}[b]{0.25\textwidth}
    \centering
    \includegraphics[width=0.95\textwidth]{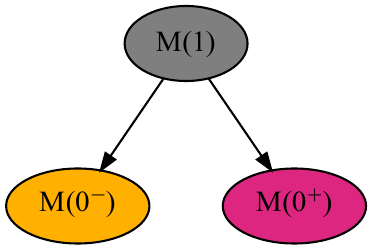}
    \caption{}
    \label{fig:ci_bistable}
  \end{subfigure}
  \caption{(a) Hasse diagram of a Morse representation of \eqref{eq:chafee_infante} at $\lambda = 28$, where the number indicates the dimension of the unstable manifold of the equilibrium. (b) The Hasse diagram of a coarsened Morse representation of \eqref{eq:chafee_infante} which only captures bistability and is the image of (a) under a poset epimorphism.}
  \label{fig:ci_MRfull}
\end{figure}

The Allen-Cahn equation was introduced as a model for the phase separation of iron alloys and thus a question of primary concern is the following: to which stable pattern $M(0^+)$ or $M(0^-)$ will a given initial condition  converge? Using the language of Morse representations we can address this question as follows. Let $M(1)$ denote the invariant set consisting of the equilibria $\setdef{M(k^\pm)}{k=1,\ldots,4}\cup\setof{M(5)}$ and the connecting orbits between them, and consider the Morse representation $\setof{M(0^-),M(0^+),M(1)}$ shown in Figure~\ref{fig:ci_bistable}.

A one-dimensional ordinary differential equation with the Morse representation in Figure~\ref{fig:ci_bistable} is
\begin{equation}
  \label{eq:bistableODE}
  \frac{dx}{dt} =  x(1-x^2)
\end{equation}
for which the origin $0$ is an unstable fixed point and $\pm1$ are stable fixed points. This suggests (as was done in Section~\ref{sec:red_coral}) learning a latent model $g_1\colon \R\to \R$. However, since the global attractor of \eqref{eq:chafee_infante} at $\lambda = 28$ is five-dimensional, a reasonable concern is that $E$ and $g$ will fail to produce quantitatively useful information. With this in mind we also seek latent maps $g_2\colon \R^2\to \R^2$ and $g_3\colon \R^3\to \R^3$.

Using a spectral method in dimension $N = 64$ to solve \eqref{eq:chafee_infante} we consider the time-$\tau$ map with $\tau = 0.1$. We generated 1,000 trajectories with $30$ steps each to construct a dataset of 30,000 state pairs. Using this dataset we train autoencoders $E_d$ and latent maps $g_d$, $d=1,2,3$, and apply \texttt{CMGDB} to $g_d$ to obtain multivalued maps $\cG_d$ using adaptive grids. The resulting Morse graphs $\sMG(\cG_d)$ are shown in Figures~\ref{fig:ci_latent_1d}(a), \ref{fig:ci_morse_graph_dynamics}(a),  and \ref{fig:ci_latent_3d}(a), respectively. Each of these Morse graphs has two minimal nodes. The sampling summarized in Table~\ref{tab:sampled_residual_tolerance} finds no violation of the local residual bound \eqref{eq:loc_semiconjugacy} when $d=2,3$. Although a finite sample cannot verify \eqref{eq:loc_semiconjugacy}, this leaves open the possibility that Theorem~\ref{thm:main_simple} is applicable, which would imply that \eqref{eq:chafee_infante} exhibits bistability. We can obtain a coarsened Morse graph from $\cG_d$ analogous to Figure~\ref{fig:ci_bistable}. The resulting coarsened Morse graph and the corresponding regions of phase space are shown for $d=2$ in Figures~\ref{fig:ci_2d_mg_coarsening} and \ref{fig:ci_2d_ms_coarsening}.

\begin{figure}[!htpb]
  \centering
  \begin{subfigure}[c]{0.43\textwidth}
    \centering
    \includegraphics[width=0.9\textwidth]{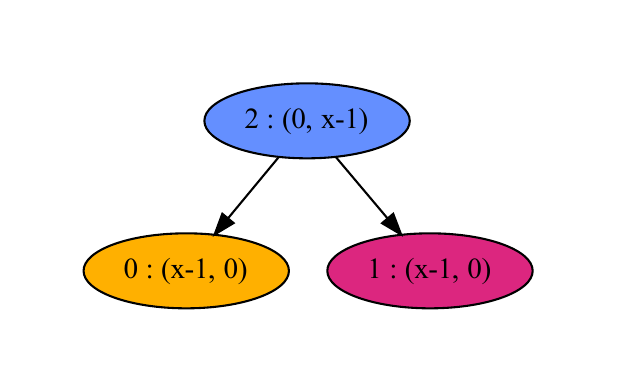}
    \caption{}
  \end{subfigure}
  \hfill
  \begin{subfigure}[c]{0.53\textwidth}
    \centering
    \includegraphics[width=\textwidth]{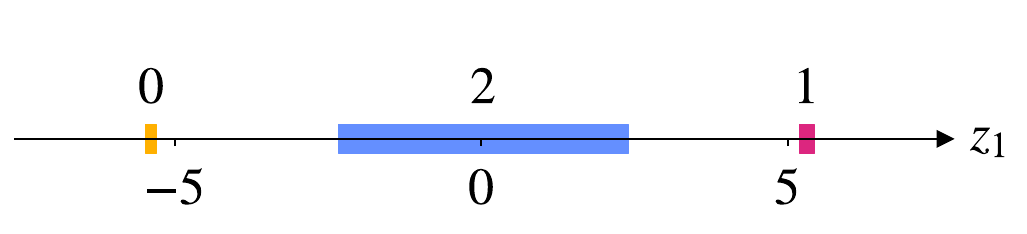}
    \caption{}
  \end{subfigure}
  \caption{(a) Hasse diagram of $\sMG(\cG_1)$. (b) Corresponding color-coded regions of phase space for the one-dimensional latent model $g_1$ for the time-$\tau$ map of \eqref{eq:chafee_infante}.}
  \label{fig:ci_latent_1d}
\end{figure}

\begin{figure}[!htpb]
  \centering
  \begin{subfigure}[b]{0.46\textwidth}
    \centering
    \includegraphics[width=0.82\textwidth]{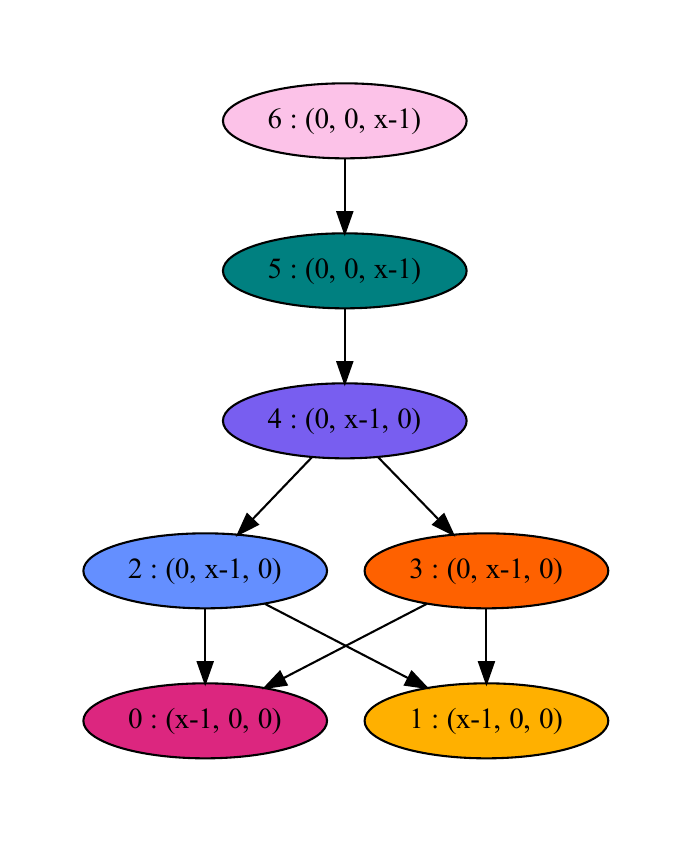}
    \caption{}
    \label{fig:ci_2d_mg}
  \end{subfigure}
  \hfill
  \begin{subfigure}[b]{0.46\textwidth}
    \centering
    \includegraphics[width=\textwidth]{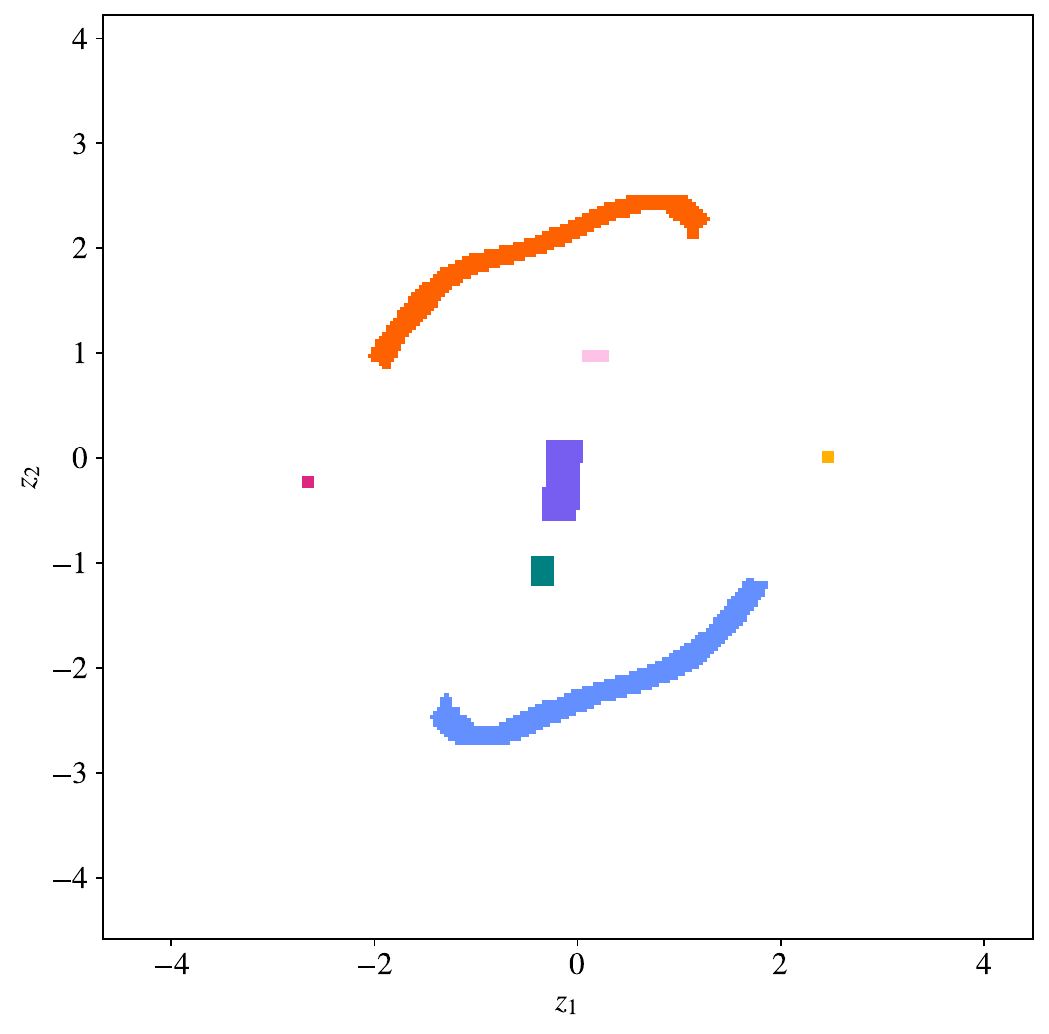}
    \caption{}
    \label{fig:ci_2d_ms}
  \end{subfigure}

  \begin{subfigure}[b]{0.46\textwidth}
    \centering
    \includegraphics[width=0.8\textwidth]{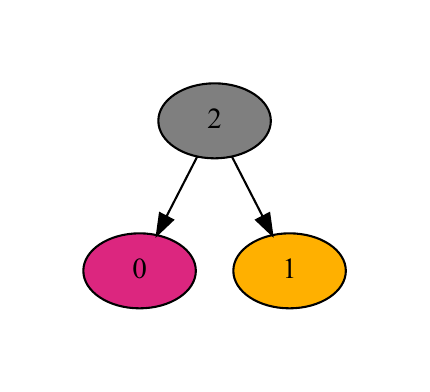}
    \caption{}
    \label{fig:ci_2d_mg_coarsening}
  \end{subfigure}
  \hfill
  \begin{subfigure}[b]{0.46\textwidth}
    \centering
    \includegraphics[width=\textwidth]{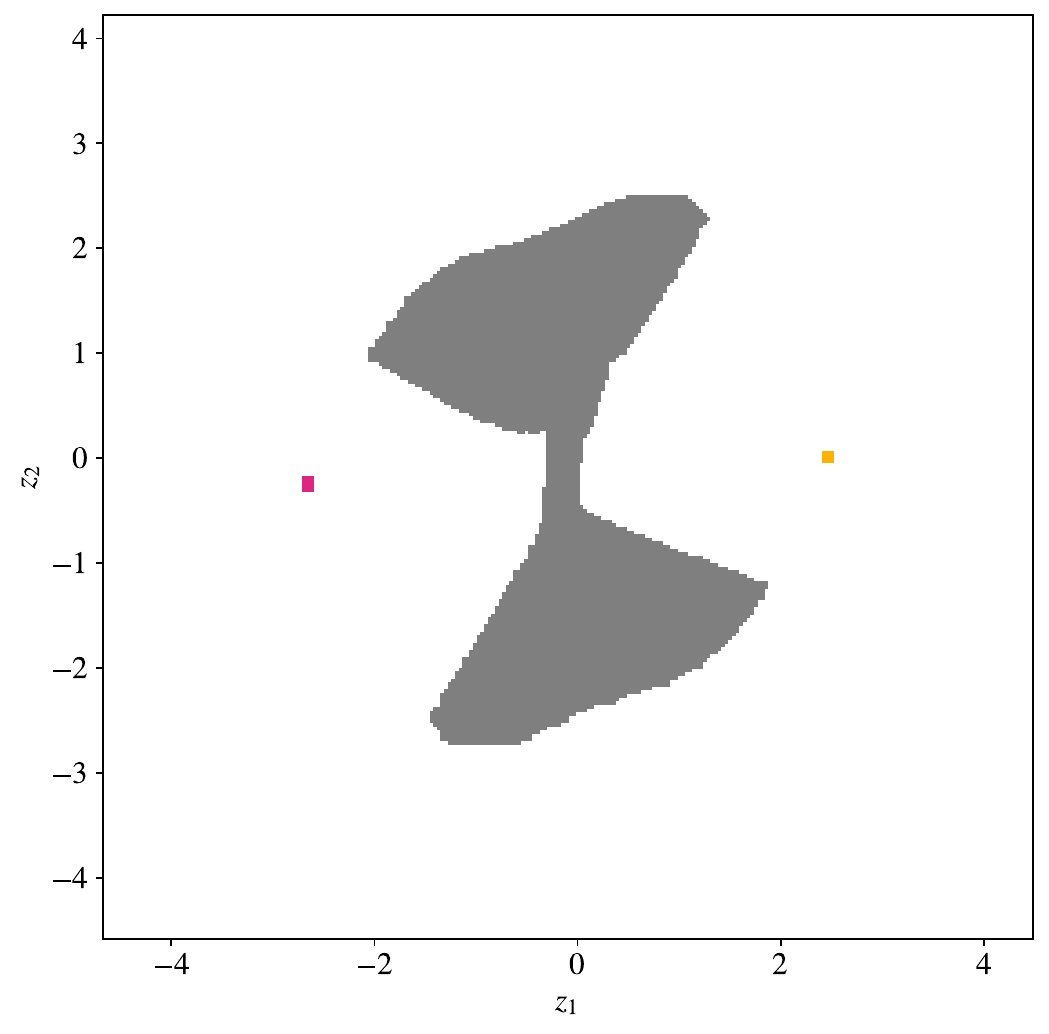}
    \caption{}
    \label{fig:ci_2d_ms_coarsening}
  \end{subfigure}
  \caption{(a) Hasse diagram of $\sMG(\cG_2)$. (b) Corresponding color-coded regions of phase space for the latent model $g_2$ for the time-$\tau$ map of \eqref{eq:chafee_infante}. (c) Hasse diagram of coarsening of $\sMG(\cG_2)$. The coarse Morse graph only captures bistability and is the image of (a) under a poset epimorphism. (d) The corresponding color-coded regions of phase space for the coarse Morse graph.}
  \label{fig:ci_morse_graph_dynamics}
\end{figure}

\begin{figure}[!htpb]
  \centering
  \begin{subfigure}[c]{0.35\textwidth}
    \centering
    \includegraphics[width=\textwidth]{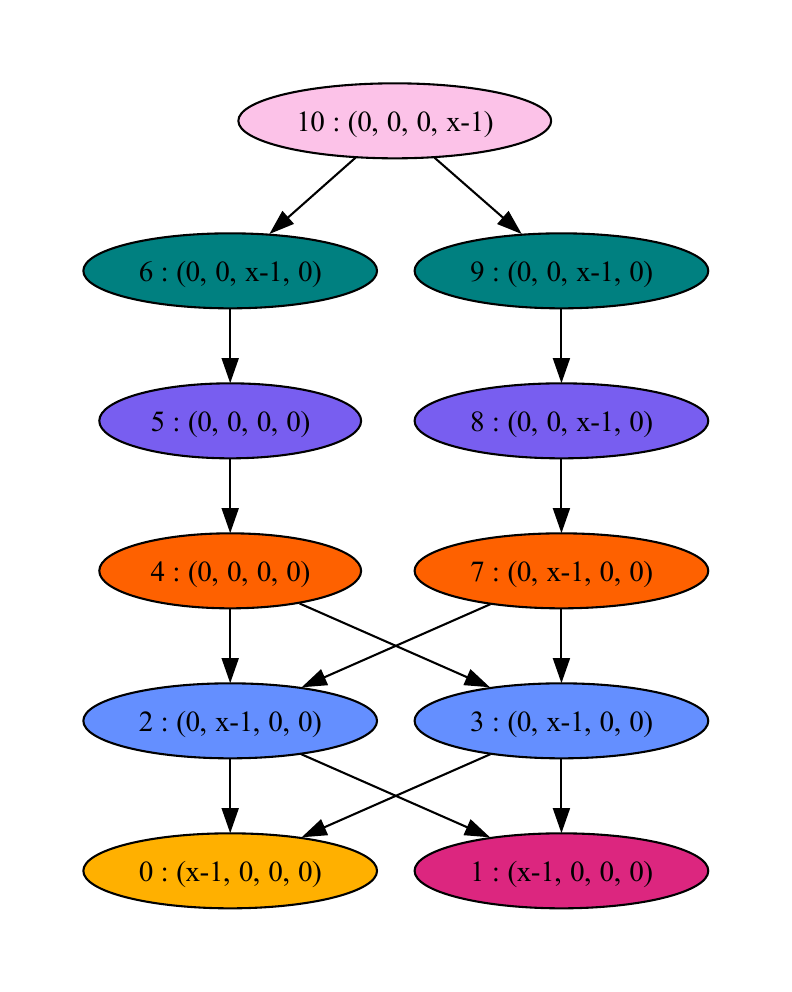}
    \caption{}
  \end{subfigure}
  \hfill
  \begin{subfigure}[c]{0.61\textwidth}
    \centering
    \includegraphics[width=\textwidth]{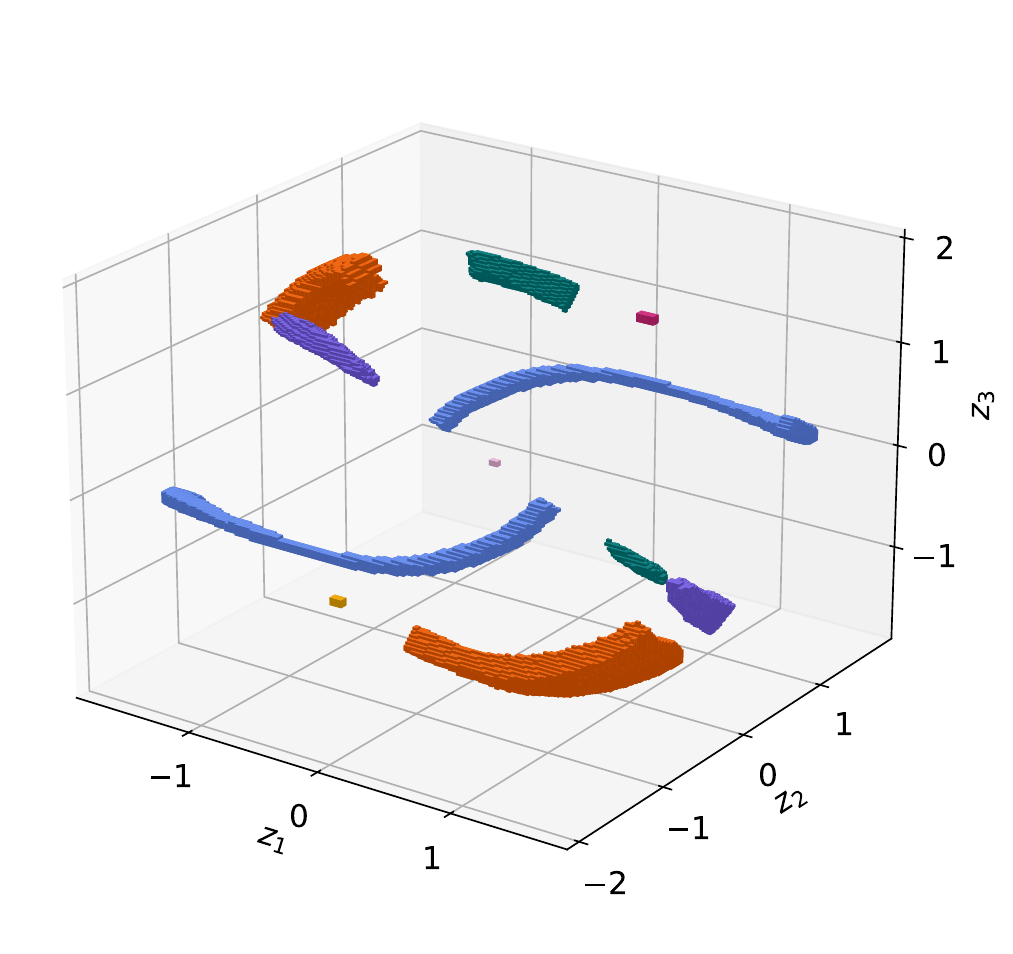}
    \caption{}
  \end{subfigure}
  \caption{(a) Hasse diagram of $\sMG(\cG_3)$. (b) Corresponding color-coded regions of phase space for the latent model $g_3$ for the time-$\tau$ map of \eqref{eq:chafee_infante}.}
  \label{fig:ci_latent_3d}
\end{figure}

To investigate the extent to which the approach can be used to classify the asymptotic behavior of a given initial condition, we generated a test dataset with $10^4$ initial conditions and ground truth labels determined by integrating \eqref{eq:chafee_infante} until convergence. The following procedure was repeated for $d=1,2,3$ and fifteen trials consisting of three training runs per dataset over five training datasets. For each trial, we obtained a distinct $E_d$ and $g_d$.

Using $g_d$, we computed a multivalued map $\tilde{\cG}_d$ on a uniform grid (it is not computationally feasible to define the regions of attraction on the adaptive grid). We then computed $\sMG(\tilde{\cG}_d)$ and the maximal region of attraction under $\tilde{\cG}_d$ of each minimal node of the Morse graph $\sMG(\tilde{\cG}_d)$. Let $R(M(0^\pm))$ denote the basin of attraction under $\tilde{\cG}_d$ of $\cM \in \sMG(\tilde{\cG}_d)$ such that  $E(M(0^\pm)) \in |\pi_{\cG}^{-1}(\cM)|$.

For example, see Figure~\ref{fig:ci_attractor_basins}, which shows the regions of attraction for $\tilde{\cG}_d$. For each initial condition $u$, we determined whether $E_d(u)$ is in $R(M(0^+))$ or $R(M(0^-))$. In these cases, $u$ is classified via $\tilde{\cG}_d$ to converge to the solution $M(0^+)$ or $M(0^-)$ respectively. If $E_d(u)$ is in the complement of $R(M(0^+)) \cup R(M(0^-))$, then we say that the final state is \emph{undetermined}.

In Table~\ref{tab:basins_attraction}, we report the outcomes averaged over the fifteen trials. The average percentage of initial conditions assigned to the wrong equilibrium is $0.585\%$ at $d=1$ and drops to $0.071\%$ at $d=3$. On the other hand, we find that the method is conservative, with an average of $56.6\%$, $42.75\%$, and $39.82\%$ of initial conditions being classified as undetermined for $d=1, 2$ and $3$ respectively. However, the percentage of initial conditions classified as undetermined decreases as the latent dimension is increased.

We attribute the conservatism of the approach to the discrepancy between the length of trajectories in the training dataset versus the typical number of steps until convergence. More specifically, for approximately $21\%$ of the initial conditions in the test set, more than $60$ steps are required for convergence, whereas we train on trajectories of $30$ steps. In Section~\ref{sec:appendix_ci_roa_statistics} we provide additional summary statistics and the results for the individual computations.

\begin{figure}[!htpb]
  \centering
  \begin{subfigure}[b]{0.47\textwidth}
    \centering
    \includegraphics[width=\textwidth]{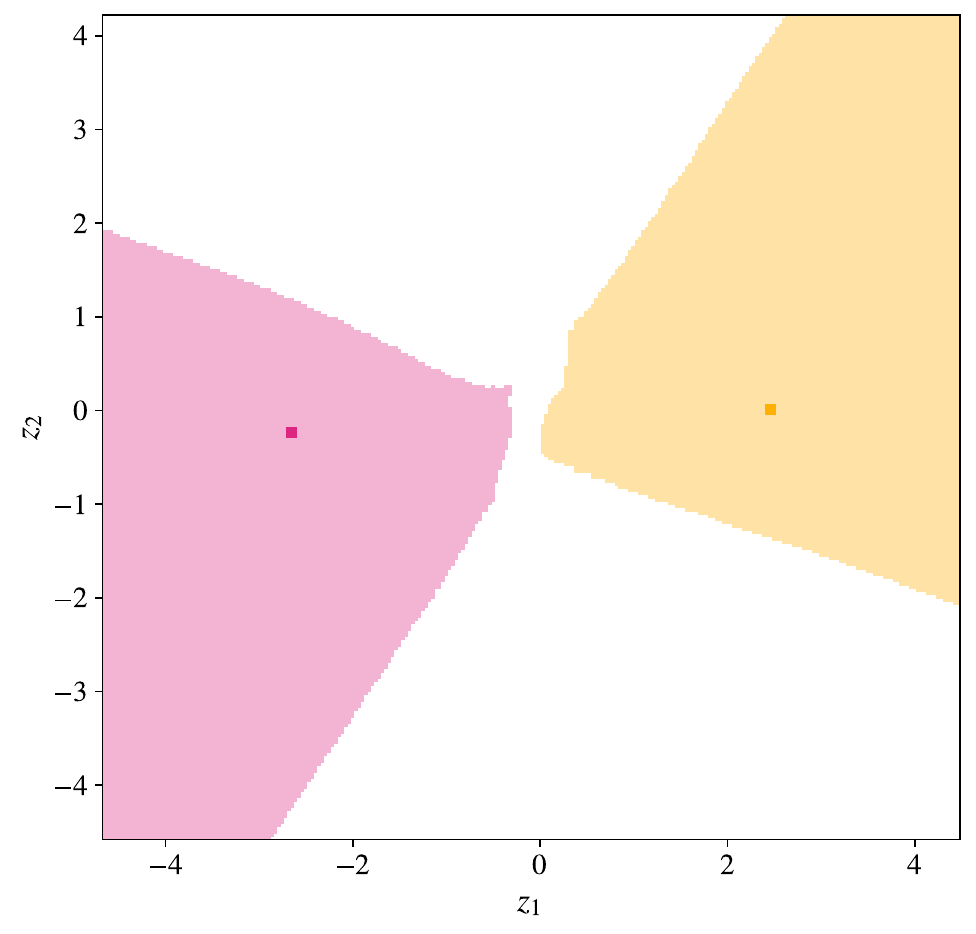}
    \caption{}
  \end{subfigure}
  \hfill
  \begin{subfigure}[b]{0.47\textwidth}
    \centering
    \includegraphics[width=\textwidth]{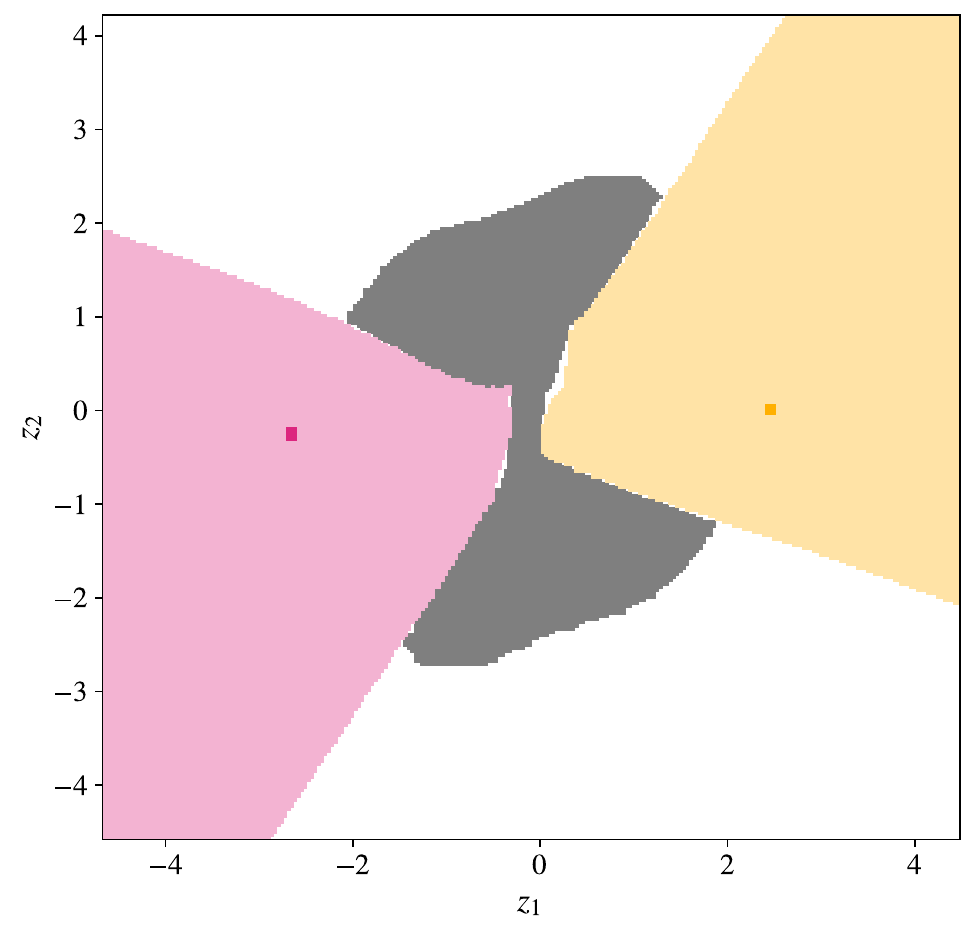}
    \caption{}
  \end{subfigure}
  \caption{(a) The
    sets $|\pi_{\cG_2}^{-1}(0)|$ (dark magenta square) and $|\pi_{\cG_2}^{-1}(1)|$ (orange square), and their corresponding basins of attraction, shown in pink and yellow respectively.
  (b) Overlay of the regions of attraction and the color-coded regions of phase space for the coarse Morse graph.}
  \label{fig:ci_attractor_basins}
\end{figure}

\begin{table}[!htpb]
  \centering
  \renewcommand{\arraystretch}{1.2}
  \begin{tabular}{@{}lccc@{}}
    \toprule
    Latent Dimension               & $d=1$ & $d=2$ & $d=3$ \\ \midrule
    Correctly classified as $M(0^-)$ & $20.65\%$ & $26.92\%$ & $\mathbf{30.68\%}$ \\
    Correctly classified as $M(0^+)$ & $22.17\%$ & $\mathbf{30.00\%}$ & $29.43\%$ \\
    Undetermined                     & $56.60\%$ & $42.75\%$ & $\mathbf{39.82}\%$ \\ \midrule
    Misclassified as $M(0^-)$        & $0.282\%$ & $0.104\%$ & $\mathbf{0.023}\%$ \\
    Misclassified as $M(0^+)$        & $0.303\%$ & $0.220\%$ & $\mathbf{0.048}\%$ \\
    \bottomrule
  \end{tabular}
  \caption{Mean percentages of initial conditions that are undetermined, and misclassified or correctly classified as converging to the steady state solutions $M(0^-)$ and $M(0^+)$ of \eqref{eq:ci_steady_state}, averaged over 15 computations for each latent dimension.}
  \label{tab:basins_attraction}
\end{table}

\section{Discussion and open questions}
\label{sec:conclusions}

The motivation for this work comes from two sources. First, the work of \cite{gameiro:gelb:mischaikow} proves that any characterization of dynamics using the language of Morse representations and the Conley index can be recovered given sufficient data and sufficient computational effort. A key step in this process is the ability to identify attracting blocks. Second, Theorem~\ref{thm:main_simple} provides an explicit quantitative condition \eqref{eq:loc_semiconjugacy} under which attracting blocks for a latent dynamical system can be lifted to the dynamics of interest. The results presented in Section~\ref{sec:applications} repeatedly show that information from the latent dynamics can be lifted even if \eqref{eq:loc_semiconjugacy} is not satisfied. We provide a potential explanation.

As indicated in Proposition~\ref{prop:MR=Att}, Morse representations of $f$ are equivalent to finite lattices of attractors, which in turn can be identified via finite lattices of attracting neighborhoods. Our computational approach seeks to identify elements of $\sABlock(f)$ by lifting elements of $\sABlock(g)$ that can be captured by $\cG$. However, as the following example shows, attracting blocks are extremely special forms of attracting neighborhoods.

\begin{ex}
  \label{ex:radial}
  Consider the following differential equation given in polar coordinates:
  \begin{equation*}
    \frac{dr}{dt} = -\epsilon r, \qquad
    \frac{d\theta}{dt} = 1.
  \end{equation*}
  The global attractor is the origin. Assume $\epsilon>0$ is small and the sampling rate is high. It is impossible, using a reasonable cubical complex, to find an attracting block $\cN$ of the origin contained in a ball of moderate radius centered at the origin. However, any neighborhood of the origin is an attracting neighborhood for the origin.
\end{ex}

In light of Example~\ref{ex:radial}, a possible explanation for our success is that even though \eqref{eq:loc_semiconjugacy} fails and thus we cannot guarantee that $E^{-1}(N)$ is an attracting block, it is possible that $E^{-1}(N)$ is an attracting neighborhood. This raises the question of whether we can improve our methods to obtain more plausible guarantees. We discuss three options, which involve altering the combinatorial model, the learned model, or the data respectively. It is worth noting that by definition we can determine whether a given set $N$ is an attracting block via a single application of the map, i.e., by checking whether $f(N)\subset \Int(N)$. Determining whether a set is an attracting neighborhood is, in general, much more difficult, though machine learning techniques for identifying attracting neighborhoods have been proposed \cite{gameiro:gelb:kalies:kramar:mischaikow:tatasciore}. Thus, we restrict our attention to algorithms that produce attracting blocks.

The first option is to consider a wider range of candidate attracting blocks while keeping the learned encoder $E$ and latent map $g$ fixed. Note that in each example we used a candidate $|\cN|$ for $\cN \in \sAtt(\cG)$.  However, typically $\sAtt(\cG)$ is a small subset of $\sInvset^+(\cG)$, and for the three-dimensional Leslie system discussed in Section~\ref{sec:3d_leslie}, the tolerances of the associated attracting blocks are within a few box widths (see Section~\ref{sec:appendix_sampled_residual_tolerance}). At fine resolutions needed to capture complicated dynamics, it can therefore be difficult to satisfy \eqref{eq:loc_semiconjugacy} for $\cN \in \sAtt(\cG)$. Thus, a coarser approach could yield more accurate results. One option, at a fixed resolution, is to search among the sets $\cN' \in \sInvset^+(\cG)$ satisfying $\omega(\cN',\cG)=\cN$ for a given $\cN \in \sAtt(\cG)$ and empirically test \eqref{eq:loc_semiconjugacy} for $|\cN'|$. The example discussed in Section~\ref{sec:3d_leslie} suggests that this coarsening can either be carried out at the grid resolution, by recomputing $\cG$ at a lower resolution, or by providing a coarser Morse graph which is the image of $\sMG(\cG)$ under an epimorphism.

The second option is to modify the learned maps rather than the combinatorial model. In a staged approach, after $\cN$ is initially determined, the residual in \eqref{eq:loc_semiconjugacy} may be locally optimized over $|\cN|$ yielding an encoder $E'$ and latent dynamics model $g'$ that are specialized for performance on the region of interest.

The third option is to modify the data through adaptive experimental design, by which we mean an iterative process of collecting data and computing Conley-Morse graphs. This requires developing a theoretically justified method for optimal sampling given Conley-Morse graph information.
In other words, given the results from an initial set of models, how should one sample additional data to produce an improved set of models? We leave these options to future work.

For all of the applications in Section~\ref{sec:applications}, we provide the Conley-Morse graph of the latent dynamics.
In many cases, the Conley index information derived from $g$ was lifted to $f$.
Future work should develop a verifiable theoretical justification for when this will occur.

\section{Data and code availability}

The data and code used for the applications are available at \url{https://github.com/begelb/latent_dynamics/tree/paper}.

\section*{Acknowledgments}
B.G. was supported by the National Science Foundation under Grant DGE-1842213. B.R. was supported by the Air Force Office of Scientific Research under award number FA9550-23-1-0400. W.K. was supported by the Air Force Office of Scientific Research under award numbers FA9550-23-1-0400 and FA9550-23-1-0011. M.G. and K.M. were partially supported by the Air Force Office of Scientific Research under award numbers FA9550-23-1-0400 and FA9550-23-1-0011.

The authors acknowledge the Office of Advanced Research Computing (OARC) at Rutgers, The State University of New Jersey for providing access to the Amarel cluster and associated research computing resources that have contributed to the results reported here. Anthropic Claude and OpenAI Codex were used to assist with code development and testing, numerical analysis and figure preparation. The numerical results and figures were produced by the documented computational pipeline.  The authors assume responsibility for all content.

\appendix
\section{Computational details}
\label{sec:appendix}

\subsection{Neural network architecture, data, and training parameters}
\label{sec:appendix_data_training}
We use a standard autoencoder framework to construct approximate semiconjugacies. The maps $E$, $D$, and $g$ are instantiated as multilayer perceptrons. In Table~\ref{tab:architecture} we report the network architecture parameters for each example.
\begin{table}[H]
  \tiny
  \centering
  \setlength{\tabcolsep}{2pt}
  \begin{tabular}{lcccc}
    \toprule
    & Extended Leslie (10D) & Three-dimensional Leslie & Chafee-Infante & Red coral \\
    \midrule
    Ambient dimension                  & 10 & 3 & 64 & 13 \\
    Latent dimension                   & 2 & 2 & $1,2,3$ & 1 \\
    Hidden layers per network          & 4 & 3 & 2 & 3 \\
    Hidden layer widths                & 64 & 32 & 32 or 64 & 64 \\
    Hidden activation                  & ReLU & ReLU & Tanh & ReLU \\
    Encoder/dynamics output activation & Tanh & Tanh & None & Tanh \\
    Decoder output activation          & Sigmoid & Sigmoid & None & Sigmoid \\
    \bottomrule
  \end{tabular}
  \caption{Network architecture parameters. For each Chafee-Infante example, the encoder, latent dynamics model, and decoder have hidden-layer widths $(64,32)$, $(32,32)$, and $(32,64)$, respectively.}
  \label{tab:architecture}
\end{table}

The training is based on a finite sample
\[
  \cD := \setof{(x_i, f(x_i))}_{i=1}^N, \quad x_i \in X,\ 1 \leq i \leq N,
\]
which is built from $T$ iterations of $m$ samples of initial conditions. In some cases, we iterate for $T$ steps and discard the first $T_0$ iterations, thus $N=m(T-T_0)$. The values for each dataset are given in Table~\ref{tab:data}.

\begin{table}[H]
  \tiny
  \centering
  \setlength{\tabcolsep}{2pt}
  \begin{tabular}{lcccc}
    \toprule
    & Extended Leslie (10D) & Three-dimensional Leslie & Chafee-Infante & Red coral \\
    \midrule
    $m$ (training)   & 8000 & 3200 & 1000 & 500 \\
    $m$ (validation) & 2000 & 800  & ---  & 10000 \\
    $N=m(T-T_0)$ (training pairs)       & 160000 & 64000 & 30000 & 10000 \\
    $T_0$ (transient steps discarded) & 0 & 10 & 0 & 0 \\
    $T$ (total trajectory length)  & 20 & 30 & 30 & 20 \\
    Sampling procedure       & Unif($X$) & Unif($X$) & Unif($X$) & Sobol$'$ \\
    Scaling procedure        & MinMax & MinMax & None & MinMax \\
    \bottomrule
  \end{tabular}
  \caption{Dataset parameters. In the scaling-procedure row, MinMax refers to \texttt{sklearn.preprocessing.MinMaxScaler} with feature range $(0,1)$ \cite{scikit-learn}. For Chafee-Infante, the initial spectral coefficients satisfy $a_{0,k}\sim\operatorname{Unif}(-2,2)e^{-0.5(k-1)}$ for $k=1,\ldots,64$. Sampling according to a Sobol$'$ sequence was performed with \texttt{scipy.stats.qmc.Sobol} with \texttt{scramble=True} \cite{roy:owen:balandat:haberland, virtanen:etal}.}
  \label{tab:data}
\end{table}
To account for varying feature scales of the data and activation function ranges, the data is typically rescaled coordinate-wise to $[0,1]$ before training. We train the neural networks by choosing the parameters of $E$, $D$, and $g$ to minimize
\[
  \frac{1}{N}\sum_{i=1}^N\Bigl(w_1L_1(x_i)+w_2L_2(x_i,f(x_i))+w_3L_3(x_i,f(x_i))\Bigr),
\]
where $w_i$ are nonnegative hyperparameters, and the functions $L_i$ are defined as
\begin{align*}
  L_1(x) &= \lVert D(E(x))-x\rVert_2^2, \\
  L_2(x,y) &= \lVert D(g(E(x)))-y\rVert_2^2,\\
  L_3(x,y) &= \lVert g(E(x))-E(y)\rVert_2^2.
\end{align*}
The minimization is carried out via the Adam optimizer~\cite{kingma:ba}, with reduction of the initial learning rate on plateau. The function $L_1$ is the standard reconstruction error. The value $L_2(x,f(x))$ measures one-step prediction error through the decoder and latent map. Finally, $L_3(x,f(x))$ is the sampled semiconjugacy loss.

In Table~\ref{tab:hyperparameters} we report the training hyperparameters for each example. Where validation data are available, we evaluate the loss after each epoch to monitor generalization and control the learning-rate schedule. Training stops according to the listed patience. The Chafee-Infante models are trained for the full number of epochs without a validation set or early stopping.

\begin{table}[H]
  \tiny
  \centering
  \setlength{\tabcolsep}{2pt}
  \begin{tabular}{lcccc}
    \toprule
    & Extended Leslie (10D) & Three-dimensional Leslie & Chafee-Infante & Red coral \\
    \midrule
    Optimizer                            & Adam & Adam & Adam & Adam \\
    Learning rate                        & $10^{-3}$ & $10^{-3}$ & $3\times 10^{-3}$ & $10^{-3}$ \\
    Batch size                           & 1024 & 1024 & 30000 & 1024 \\
    Max epochs                           & 1000 & 1000 & 4000 & 1000 \\
    Early stopping patience              & 100 & 100 & --- & 100 \\
    Reconstruction loss weight ($w_1$)   & 100 & 10 & 1 & 10 \\
    Prediction loss weight ($w_2$)       & 10 & 10 & 1 & 10 \\
    Semiconjugacy loss weight ($w_3$)    & 20 & 1 & 0 & 1 \\
    \bottomrule
  \end{tabular}
  \caption{Training hyperparameters.}
  \label{tab:hyperparameters}
\end{table}

Finally, Table~\ref{tab:coral_data} lists the demographic parameters $b_i$ and $s_i$ of the red coral model \eqref{eq:coral}.
\begin{table}[h]
  \centering
  \begin{tabular}{ccc}
    $i$ & $b_i$ & $s_i$ \\ \hline
    1  & 0       & 0.889 \\
    2  & 0       & 0.633 \\
    3  & 2.89    & 0.697 \\
    4  & 10.03   & 0.517 \\
    5  & 21.59   & 0.437 \\
    6  & 39.02   & 0.287 \\
    7  & 56.41   & 0.571 \\
    8  & 77.72   & 0.333 \\
    9  & 103.23  & 0.750 \\
    10 & 131.87  & 1.000 \\
    11 & 164.57  & 0.333 \\
    12 & 201.46  & 1.000 \\
    13 & 242.65  & ---   \\ \hline
  \end{tabular}
  \caption{Demographic parameters for \eqref{eq:coral} obtained from \cite{santangelo:bramanti:iannelli}.}
  \label{tab:coral_data}
\end{table}

\subsection{Morse graph computation}
\label{sec:appendix_CMGDB}
We compute the Morse graph and Conley indices for the latent dynamics with the software package \texttt{CMGDB}~\cite{CMGDB}, which realizes the combinatorial constructions of Section~\ref{sec:background}. The input to \texttt{CMGDB} consists of the latent map $g \colon Z \to Z$ together with a hyperrectangle $B \subset Z$. In our examples, $B$ encloses the encoded inputs and outputs of the training pairs (appropriately scaled). If $d = \dim Z$, define
\begin{align*}
  \mathcal E_j &:= \setdef{E(x)_j}{x=x_i\ \text{or}\ x=f(x_i),\ (x_i,f(x_i))\in\cD},\\
  a_j &:= \min\mathcal E_j,\qquad b_j:=\max\mathcal E_j.
\end{align*}
Then
\[
  B = \prod_{j=1}^{d} \bigl[\, a_j - \Delta(b_j - a_j),\ b_j + \Delta(b_j - a_j) \,\bigr],
\]
where the padding factor $\Delta \geq 0$ enlarges each interval by a fraction $\Delta$ of its length.

The software \texttt{CMGDB} involves a discretization of $B$ into a cubical complex $\cX$. To indicate that the multivalued map used by \texttt{CMGDB} is derived from the map $g$ in the latent space we denote it by $\cG \colon \cX \mvmap \cX$. The geometric representation of every cell $\xi \in \cX$ is a hyperrectangle $|\xi|$. The image of this hyperrectangle $|\xi|$ under $g$ is approximated by a bounding box of the images under $g$ of all its corner points. The image $\cG(\xi)$ is the set of cells in $\cX$ with geometric realizations intersecting the bounding box approximating $g(|\xi|)$, padded by a layer of boxes to account for approximation errors. This construction is heuristic since corner evaluations and padding do not certify that $\cG$ is an outer approximation of $g$. Consequently, the computed Morse graph has the stated interpretation only when $g$ is a selector of $\cG$.

The cubical complex is refined adaptively in a recursive manner by subdividing only the boxes that belong to a recurrent component of the current combinatorial model~\cite{bush:gameiro:harker:kokubu:mischaikow:obayashi:pilarczyk}. There are four parameters that govern the refinement:
\begin{itemize}
  \item the \emph{initial subdivision} sets the depth of the starting grid,
  \item the \emph{minimum} and \emph{maximum subdivision} bound the adaptive resolution, and
  \item the \emph{subdivision limit} caps the number of boxes a Morse set may contain before it is refined further.
\end{itemize}

In Table~\ref{tab:cmgdb} we list the \texttt{CMGDB} parameters for each example. For all computations we set the subdivision limit equal to $10^4$.

\begin{table}[H]
  \scriptsize
  \centering
  \setlength{\tabcolsep}{3pt}
  \begin{tabular}{@{}p{0.36\textwidth}cccc@{}}
    \toprule
    Computation & $d$ & \shortstack{Domain\\padding $\Delta$} & \shortstack{Box-map\\padding} & \shortstack{Subdivision\\$(s_{\rm init},s_{\rm min},s_{\rm max})$} \\
    \midrule
    Extended Leslie (10D), latent & 2 & 0.01 & Yes & $(27,29,30)$ \\
    Three-dimensional Leslie, $g_1$ fine & 2 & 0.01 & Yes & $(23,23,27)$ \\
    Three-dimensional Leslie, $g_1$ coarse & 2 & 0.01 & Yes & $(22,22,24)$ \\
    Chafee-Infante, $d=1$ & 1 & 0.1 & Yes & $(7,8,11)$ \\
    Chafee-Infante, $d=2$ & 2 & 0.1 & Yes & $(14,16,22)$ \\
    Chafee-Infante, $d=3$ & 3 & 0.1 & Yes & $(21,24,33)$ \\
    Red coral, latent & 1 & 0.01 & Yes & $(8,8,12)$ \\
    Two-dimensional Leslie, reference & 2 & --- & No & $(24,27,28)$ \\
    Three-dimensional Leslie, reference & 3 & --- & No & $(29,33,36)$ \\
    \bottomrule
  \end{tabular}
  \caption{\texttt{CMGDB} parameters.}
  \label{tab:cmgdb}
\end{table}

\subsection{Chafee-Infante statistics on regions of attraction}
\label{sec:appendix_ci_roa_statistics}
In this section, we provide details on the region of attraction computations discussed in Section~\ref{sec:chafee_infante}. All $45$ computations resulted in Morse graphs with two distinct minimal nodes $j, k$ such that $E(M(0^+)) \in |\pi_{\cG}^{-1}(j)|$ and $E(M(0^-)) \in |\pi_{\cG}^{-1}(k)|$. For these computations, we used a uniform grid with eight subdivisions per dimension. Thus each coordinate dimension is subdivided into $2^8=256$ intervals, and the full grid has $256^d$ cells. The domain padding was $\Delta=0.1$.

In Table~\ref{tab:basins_attraction_raw} we report the results for each run, and in Table~\ref{tab:basins_attraction_summary} we provide summary statistics. Across all runs, the mean percentage correctly classified as $M(0^\pm)$ $\pm$ SD was $42.82\% \pm 17.89\%$ for $d=1$, $56.93\% \pm 11.36\%$ for $d=2$, and $60.11\% \pm 7.77\%$ for $d=3$. We highlight that for $d=3$, the spread can be mainly attributed to models trained using the first dataset; the mean $\pm$ SD excluding these runs is $62.46\% \pm 1.12\%$.

\begin{table}[!htpb]
  \centering
  \renewcommand{\arraystretch}{1.2}
  \tiny
  \begin{tabular}{@{}cccccc@{}}
    \toprule
    $d$ & Dataset & Run & Undetermined & Misclassified & Correctly classified as $M(0^\pm)$ \\ \midrule
    $d=1$ & 1 & 1 & 33.27\% & 0.950\% & 65.78\% \\
    & 1 & 2 & 31.92\% & 01.12\% & 66.96\% \\
    & 1 & 3 & 63.09\% & 0.380\% & 36.53\% \\
    & 2 & 1 & 53.20\% & 0.430\% & 46.37\% \\
    & 2 & 2 & 51.04\% & 0.790\% & 48.17\% \\
    & 2 & 3 & 55.17\% & 0.530\% & 44.30\% \\
    & 3 & 1 & 55.05\% & 0.690\% & 44.26\% \\
    & 3 & 2 & 48.91\% & 01.11\% & 49.98\% \\
    & 3 & 3 & 52.42\% & 0.680\% & 46.90\% \\
    & 4 & 1 & 46.72\% & 0.470\% & 52.81\% \\
    & 4 & 2 & 35.60\% & 0.800\% & 63.60\% \\
    & 4 & 3 & 73.48\% & 0.280\% & 26.24\% \\
    & 5 & 1 & 96.91\% & 0.070\% & 03.02\% \\
    & 5 & 2 & 69.24\% & 0.220\% & 30.54\% \\
    & 5 & 3 & 82.92\% & 0.250\% & 16.83\% \\
    \addlinespace
    $d=2$ & 1 & 1 & 34.40\% & 0.440\% & 65.16\% \\
    & 1 & 2 & 66.35\% & 0.060\% & 33.59\% \\
    & 1 & 3 & 36.95\% & 0.090\% & 62.96\% \\
    & 2 & 1 & 34.02\% & 0.200\% & 65.78\% \\
    & 2 & 2 & 33.72\% & 0.280\% & 66.00\% \\
    & 2 & 3 & 33.26\% & 0.300\% & 66.44\% \\
    & 3 & 1 & 64.71\% & 0.040\% & 35.25\% \\
    & 3 & 2 & 33.68\% & 0.570\% & 65.75\% \\
    & 3 & 3 & 35.01\% & 0.300\% & 64.69\% \\
    & 4 & 1 & 49.04\% & 0.660\% & 50.30\% \\
    & 4 & 2 & 49.43\% & 0.570\% & 50.00\% \\
    & 4 & 3 & 49.41\% & 0.280\% & 50.31\% \\
    & 5 & 1 & 50.01\% & 0.640\% & 49.35\% \\
    & 5 & 2 & 35.57\% & 0.150\% & 64.28\% \\
    & 5 & 3 & 35.70\% & 0.280\% & 64.02\% \\
    \addlinespace
    $d=3$ & 1 & 1 & 36.53\% & 0.380\% & 63.09\% \\
    & 1 & 2 & 43.77\% & 0.150\% & 56.08\% \\
    & 1 & 3 & 67.08\% & 0.010\% & 32.91\% \\
    & 2 & 1 & 37.38\% & 0.080\% & 62.54\% \\
    & 2 & 2 & 38.43\% & 0.020\% & 61.55\% \\
    & 2 & 3 & 38.65\% & 0.040\% & 61.31\% \\
    & 3 & 1 & 37.01\% & 0.060\% & 62.93\% \\
    & 3 & 2 & 38.25\% & 0.000\% & 61.75\% \\
    & 3 & 3 & 38.35\% & 0.000\% & 61.65\% \\
    & 4 & 1 & 34.91\% & 0.020\% & 65.07\% \\
    & 4 & 2 & 37.54\% & 0.030\% & 62.43\% \\
    & 4 & 3 & 37.90\% & 0.040\% & 62.06\% \\
    & 5 & 1 & 35.74\% & 0.090\% & 64.17\% \\
    & 5 & 2 & 37.71\% & 0.080\% & 62.21\% \\
    & 5 & 3 & 38.05\% & 0.060\% & 61.89\% \\
    \bottomrule
  \end{tabular}
  \caption{Percentage of initial conditions that are undetermined, misclassified, or correctly classified as converging to one of the steady state solutions $M(0^-)$ or $M(0^+)$ of \eqref{eq:ci_steady_state}. In comparison to Table~\ref{tab:basins_attraction}, we combine the counts of misclassified and correctly classified for both solutions.}
  \label{tab:basins_attraction_raw}
\end{table}

\begin{table}[!htpb]
  \centering
  \renewcommand{\arraystretch}{1.2}
  \small
  \begin{tabular}{@{}llccc@{}}
    \toprule
    $d$ & Statistic & Undetermined & Misclassified & Correctly classified as $M(0^\pm)$ \\ \midrule
    $d=1$ & Mean & 56.60\% & 0.585\% & 42.82\% \\
    & Median & 53.20\% & 0.530\% & 46.37\% \\
    & Std & 18.17\% & 0.326\% & 17.89\% \\
    \midrule
    $d=2$ & Mean & 42.75\% & 0.324\% & 56.93\% \\
    & Median & 35.70\% & 0.280\% & 64.02\% \\
    & Std & 11.39\% & 0.208\% & 11.36\% \\
    \midrule
    $d=3$ & Mean & 39.82\% & 0.071\% & 60.11\% \\
    & Median & 37.90\% & 0.040\% & 62.06\% \\
    & Std & 07.79\% & 0.095\% & 07.77\% \\
    \bottomrule
  \end{tabular}
  \caption{Summary statistics of the percentages of initial conditions that are undetermined, misclassified, or correctly classified as converging to one of the steady state solutions $M(0^-)$ or $M(0^+)$ of \eqref{eq:ci_steady_state}. In comparison to Table~\ref{tab:basins_attraction}, we combine the counts of misclassified and correctly classified for both solutions. The statistics are computed across the 15 training runs for each latent dimension $d$.}
  \label{tab:basins_attraction_summary}
\end{table}

\subsection{Sampled residual and tolerance estimates}
\label{sec:appendix_sampled_residual_tolerance}

For each reported minimal node $q$, $\cG(\pi^{-1}(q))\subset\pi^{-1}(q)$, so $N_q=|\pi^{-1}(q)|$ is an attracting block for any selector of $\cG$. Since the box map is constructed heuristically, this does not certify that $g$ is a selector or that $N_q$ is an attracting block for $g$. We nevertheless use $N_q$ as a candidate block and numerically evaluate
\[
  \widehat R_q=\max_{x\in\mathcal S_q}d_Z\bigl(g(E(x)),E(f(x))\bigr)
  \quad\text{and}\quad
  \widehat\tau_q=\min_{z\in\mathcal T_q}\operatorname{dist}_{Z}\bigl(g(z),Z\setminus\Int(N_q)\bigr),
\]
where $\mathcal S_q$ consists of sampled states $x$ such that $E(x)\in N_q$, and $\mathcal T_q$ consists of samples from every box $|\xi|$, $\xi\in\pi^{-1}(q)$. More specifically, for $\mathcal S_q$ we use the training and validation data, new trajectories of $f$ with initial conditions from a scrambled Sobol$'$ sequence in $X$, and decoded points $D(z)$ for sampled $z\in N_q$, together with small perturbations of them. For the red coral model, we also sample near the stable fixed points $a_0$ and $a_1$. For $\mathcal T_q$, in each box we evaluate $\operatorname{dist}_{Z}\bigl(g(z),Z\setminus\Int(N_q)\bigr)$ at the vertices, the center, and points from two scrambled Sobol$'$ sequences. We then minimize this distance within the boxes where the sampled values are smallest, and include the points found in $\mathcal T_q$. Thus $\widehat R_q$ is a lower estimate of the residual supremum, while $\widehat\tau_q$ is an upper estimate of the true tolerance.

\begin{table}[H]
  \centering
  \resizebox{\textwidth}{!}{%
    \begin{tabular}{@{}llcccc@{}}
      \toprule
      Example & $q$ & $|\mathcal S_q|$ & $\widehat R_q$ & $\widehat\tau_q$ & $\widehat R_q<\widehat\tau_q$? \\
      \midrule
      Three-dimensional Leslie, fine resolution & $0$ & $5.68\times10^5$ & $1.07$                  & $4.25\times10^{-4}$ & No \\
      & $1$ & $3.78\times10^6$ & $6.97\times10^{-1}$ & $4.06\times10^{-4}$ & No \\
      & $4$ & $2.84\times10^4$ & $2.31\times10^{-1}$ & $4.62\times10^{-4}$ & No \\
      Three-dimensional Leslie, coarse resolution & $0$ & $9.04\times10^5$ & $1.07$ & $8.01\times10^{-4}$ & No \\
      & $1$ & $3.07\times10^6$ & $6.97\times10^{-1}$ & $7.92\times10^{-4}$ & No \\
      Extended Leslie (10D) & $0$ & $2.24\times10^6$ & $6.80\times10^{-2}$ & $5.20\times10^{-5}$ & No \\
      & $1$ & $1.31\times10^5$ & $5.31\times10^{-2}$ & $5.41\times10^{-5}$ & No \\
      Red coral             & $0$ & $1.94\times10^6$ & $5.40\times10^{-2}$ & $7.79\times10^{-3}$ & No \\
      & $1$ & $7.16\times10^5$ & $2.48\times10^{-1}$ & $7.96\times10^{-3}$ & No \\
      Chafee-Infante, $d=1$ & $0$ & $1.33\times10^5$ & $6.58$                  & $1.04\times10^{-1}$ & No \\
      & $1$ & $1.31\times10^5$ & $6.11$                  & $6.58\times10^{-2}$ & No \\
      Chafee-Infante, $d=2$ & $0$ & $1.11\times10^5$ & $3.52\times10^{-2}$ & $3.95\times10^{-2}$ & \textbf{Yes} \\
      & $1$ & $1.25\times10^5$ & $1.60\times10^{-2}$ & $4.25\times10^{-2}$ & \textbf{Yes} \\
      Chafee-Infante, $d=3$ & $0$ & $1.19\times10^5$ & $4.31\times10^{-3}$ & $2.34\times10^{-2}$ & \textbf{Yes} \\
      & $1$ & $1.18\times10^5$ & $4.73\times10^{-3}$ & $2.36\times10^{-2}$ & \textbf{Yes} \\
      \bottomrule
    \end{tabular}
  }
  \caption{Sampled residual and tolerance estimates for the minimal nodes. For each $q$, $|\mathcal{T}_q| \geq 2^{23}$.}
  \label{tab:sampled_residual_tolerance}
\end{table}

For each minimal node $q$ in the Leslie, red coral, and Chafee-Infante with $d=1$ examples, the estimates satisfy $\widehat R_q \geq \widehat \tau_q$, and therefore \eqref{eq:loc_semiconjugacy} fails for $|\pi^{-1}(q)|$. For both minimal nodes in the Chafee-Infante examples with $d=2,3$, however, $\widehat R_q < \widehat \tau_q$, so the sampling finds no violation of \eqref{eq:loc_semiconjugacy}.

\clearpage
\printbibliography

\end{document}